\documentclass[reqno]{amsart}
\usepackage{amsmath,amsthm, amscd, amssymb, amsfonts, mathrsfs,graphics,enumitem}
\usepackage[plainpages=false,pdfpagelabels,
hypertexnames=false]{hyperref}
\usepackage[all]{xy}
\usepackage{color}

\usepackage{hhline}

\usepackage{eucal}

\newcommand{\yd}[1]{{}_{\ku #1}^{\ku #1}\mathcal{YD}}
\newcommand{\ydg}{{}_{\ku G}^{\ku G}\mathcal{YD}}

\def\Si{\mathbb{S}_3}

\newcommand{\mf}[1]{\mathsf{#1}}

\newcommand{\Imm}{\operatorname{Im}}
\newcommand{\Ker}{\operatorname{Ker}}

\newcommand{\id}{\operatorname{id}}
\newcommand{\h}{\operatorname{H}}
\def\T3{\mathbb{T}_3}

\def\fij{\mf{x}^{i}\mf{z}^{2k}}
\def\gij{\mf{y}^{i}\mf{z}^{2k}}

\newcommand{\Ou}[2]{\Omega^{#1}\big(#2\big)}

\newcommand\B{\mathfrak B }

\newcommand{\gr}{\operatorname{gr}}

\newcommand{\R}{{\mathcal R}}

\newcommand{\ku}{{\Bbbk}}

\newcommand{\N}{{\mathbb N}}

\def\a{\mf{a}}\def\b{\mf{b}}\def\cc{\mf{c}}\def\d{\mf{d}}

\newcommand{\cA}{\mathcal{A}}
\newcommand{\cB}{\mathcal{B}}

\newcommand{\cO}{\mathbb{T}_3}
\newcommand{\cS}{\mathcal{S}}
\newcommand{\cE}{\mathcal{E}}

\newcommand{\cX}{\mathcal{X}}

\newcommand{\End}{\operatorname{End}}

\newcommand{\Ext}{\operatorname{Ext}}

\newcommand\sgn{\operatorname{sgn}}

\newcommand\Hom{\operatorname{Hom}}

\newcommand\Tot{\operatorname{Tot}}

\theoremstyle{plain}

\newtheorem{lema}{Lemma}[section]
\newtheorem{theorem}[lema]{Theorem}
\newtheorem{cor}[lema]{Corollary}

\newtheorem{prop}[lema]{Proposition}

\newtheorem{question-app}{Question}

\theoremstyle{definition}

\newtheorem{exa}[lema]{Example}

\theoremstyle{remark}
\newtheorem{obs}[lema]{Remark}

\def\cFK{\mathcal{FK}}

\newtheoremstyle{mystyle}{2mm}{0mm}{}{}{\bfseries}{}{ }{\thmnumber{#2}.\thmnote{
#3}}
\theoremstyle{mystyle}
\newtheorem{fact}[lema]{}

\newcommand{\ot}{\otimes}
\newcommand\restr[2]{{
  \left.\kern-\nulldelimiterspace 
  #1 
  \vphantom{\big|} 
  \right|_{#2} 
  }}

\def\pf{\begin{proof}}
\def\epf{\end{proof}}

\theoremstyle{remark}

\newcommand{\lD}{\delta}
\newcommand{\aba}{p_{aba}}
\newcommand{\ab}{p_{ab}}
\newcommand{\ba}{p_{ba}}
\newcommand{\as}{p_a}
\newcommand{\bs}{p_b}
\def\La{\Lambda}

\newcommand{\lp}{\boldsymbol{p}}
\newcommand{\llq}{\boldsymbol{q}}
\newcommand{\lx}{\boldsymbol{x}}

\begin{document}

\title[The cohomology ring of the 12-dimensional Fomin-Kirillov algebra]{The cohomology ring of the 12-dimensional Fomin-Kirillov algebra}

\author[D. \c  Stefan and C. Vay]{Drago\c s \c Stefan and Cristian Vay}

\address{University of Bucharest, Faculty of Mathematics and Computer Science, 14 Academiei Street, Bucharest
Ro-010014, Romania} \email{dragos.stefan@fmi.unibuc.ro}

\address{FaMAF-CIEM (CONICET), Universidad Nacional de C\'ordoba, Medina A\-llen\-de s/n, Ciudad Universitaria, 5000 C\' ordoba, Rep\'ublica Argentina.} \email{vay@famaf.unc.edu.ar}

\begin{abstract}
The  $12$-dimensional Fomin-Kirillov algebra $\cFK_3$ is defined as the quadratic algebra with generators $a$, $b$ and $c$ which satisfy the relations $a^2=b^2=c^2=0$ and  $ab+bc+ca=0=ba+cb+ac$. By a result of A. Milinski and H.-J. Schneider, this algebra is isomorphic to the Nichols algebra associated to the Yetter-Drinfeld module  $V$, over the symmetric group $\mathbb{S}_3$, corresponding to the conjugacy class of all transpositions and the sign representation. Exploiting this identification, we compute the cohomology ring $\Ext_{\cFK_3}^*(\ku,\ku)$, showing that it is a polynomial ring $S[X]$ with coefficients in the symmetric braided algebra of $V$. As an application we also compute the cohomology rings of the bosonization $\cFK_3\#\ku\mathbb{S}_3$ and of its dual, which are $72$-dimensional ordinary Hopf algebras. 
\end{abstract} 

\noindent\keywords{Nichols bialgebras, Fomin-Kirillov algebras, Yoneda ring}
\subjclass[2010]{18G15(primary), and 16T05 (secondary)}
\maketitle

\section{Introduction}
In \cite{fominkirilov},  Fomin and Kirillov introduced a new class of quadratic algebras $\big\{\cFK_n\big\}_{n\geq 3}$, as an attempt to give a combinatorial explanation of the fact that, with respect to the linear basis on the cohomology ring of a flag manifold  consisting of all Schubert cells,  the constant structures  are non-negative.

Fomin-Kirillov algebras have many intersting properties. For example, they are strongly related to Nichols algebras. Let us recall that, by definition, a Nichols algebra is an $\N$-graded bialgebra in the category of Yetter-Drinfeld $H$-modules, which is connected and generated by $1$-degree elements. Here $H$ may be any Hopf algebra over a field $\ku$, but in this paper we are interested in the case when $H$ is the group algebra $\ku G$ of a finite group $G$. Nichols algebras over abelian groups play a fundamental role in the classification of finite dimensional pointed Hopf algebras with commutative coradical, cf.  \cite{AS} and the references therein. They have also been used as a very important tool for the investigation of certain quantum groups. 
For a non-abelian group $G$, the structure of Nichols algebras over $\ku G$ is not so well understood, but nowadays this is a very dynamic  research area, see \cite{AFGV, AG, AG2, H, HLV, HV}. One of the main difficulties in the non-commutative setting is that only few finite dimensional Nichols algebras are known. 

Turning back to Fomin-Kirillov algebras, it was proved in \cite{MS} that $\cFK_n$ is a graded bialgebra in the category of Yetter-Drinfeld $\ku\mathbb{S}_n$-modules, and it is a Nichols algebra for $n=3,4$. A similar result was obtained for $n=5$ in \cite{grania}. Based on these results, Milinski and Schneider have conjectured that  $\cFK_n$ is a Nichols algebra too, for all other values of $n$. See also \cite[Conjecture 9.3]{fominkirilov}. Notably, for $n\leq 5$, it was proved that $\dim\cFK_n<\infty$ and it was conjectured that $\dim\cFK_n=\infty$, provided that $n\geq 6$.  

The theory of support and rank variety for Hopf algebras is the driving force behind the results concerning the finite generation of the cohomology ring of these and other related structures. Rather recently, the famous result due to Golod, Venkov and Evans \cite{golod,venkov, evans}, stating that the cohomology ring of a finite group is finitely generated, has been extended for several classes of Hopf algebras, including the finite dimensional cocommutative Hopf algebras 
 \cite{friedlandersuslin}, the finite dimensional pointed Hopf algebras with abelian coradical of dimension relatively prime to $210$ and the corresponding Nichols algebras \cite{MPSW}, small quantum groups \cite{ginzburgkumar} and finite quantum function algebras \cite{gordon}. 
 
 According to our best knowledge, analogous results are not known for  pointed Hopf algebras over non-abelian groups, excepting the case when the corresponding Nichols algebras are of diagonal type, as in \cite{MPSW}. Nevertheless, using the \texttt{GAP} package \texttt{QPA} the dimension of $\Ext_{\cFK_3}^n(\ku,\ku)$ was computed for $n$ up to $40$, the outcome suggesting that  the Hilbert series of the cohomology ring $E(\cFK_3)$  is 
\[
 H_{\cFK_3}=\frac{(1+t)(1+t+t^2)}{(1-t)(1-t^4)}
\]
and that $E(\cFK_3)$ is generated by three cohomology classes of degree $1$ and one of degree $4$, see \cite{solberg}.

The aim of this paper is to clarify the ring structure of $E(\cFK_3)$ and, in particular, to confirm the computer calculations that we mentioned above.  Roughly speaking, in our main result, Theorem \ref{te:main}, we prove that  $E(\cFK_3)$ is a polynomial ring $S[X]$ over the braided symmetric algebra $S=S(V)$ of the Yetter-Drinfeld $\ku\Si$-module $V$ that help us to regard $\cFK_3$ as a Nichols algebra $\mathfrak{B}(V)$.

The key steps of the proof of Theorem \ref{te:main} are the following. First we notice that, by \cite{MS}, the algebra $\cFK_3$ is a twisted tensor product $A\ot_\sigma R$, where $A$ and $R$ are presented by generators and relations as follows: $A=\big\langle a,b\mid a^2,b^2, aba-aba\big\rangle$ and  $R=\big\langle c\mid c^2\big\rangle$, see \S\ref{subsec:B is a twist}. 

This property allows us to use a  version of Cartan-Eilenberg spectral sequence, that we construct for any graded twisting product $B=A\ot_\sigma R$ with invertible twisting map $\sigma$ that satisfying an extra mild condition. By definition, its second page is given by $E^{p,q}_2=\Ext_R^p\big(\ku, E_q(A)\big)$ and it converges to $E(B)$, cf. Theorem \ref{te: CE} and Corollary \ref{co:CE}. In order to use this spectral sequence, we have to explicit the ring structure of the cohomology ring $E(A)$ and the $R$-action on $E(A)$. In the special case of the Fomin-Kirillov algebra $\mathcal{FK}_3$, this is done in the third section, see Theorem \ref{thm:EA algebra structure} and Theorem \ref{thm:R action}. 

As an important step, we find lower and upper bounds for the dimension of $E_n(\cFK_3)$. To be more specific, by Proposition \ref{prop:La B embeds into EB} the subalgebra generated by $E_1(\cFK_3)$ is precisely $S$, the braided symmetric algebra of the Yetter-Drinfeld module $V$. Thus, the dimension of $E_n(\cFK_3)$ is at least the dimension of $S_n$. On the other hand, using a basic property of spectral sequences, we deduce that the dimension of $E_n(\cFK_3)$ is at most $N_n=\sum_{p=0}^n\dim E^{p,n-p}_2$.  For $n\leq 3$ the lower and the upper bounds of $\dim E_n(\cFK_3)$ coincide and $N_4=\dim S_4+1$. Furthermore, there is a $4$-cohomology class which is a central element and does not belong to $S_4$. All these results imply that our spectral sequence degenerates and $\dim E_n(\cFK_3)$ reaches the upper bound,  see Theorem \ref{thm:dimE^n}.

To give a presentation  of $E(\cFK_3)$ by generators and relations we use that the cohomology ring of a Nichols algebra is braided commutative  \cite[Corollary 3.13]{MPSW}. Taking into account this result and the fact that $E(\cFK_3)$ is generated by elements of degree at most $4$ (see Proposition \ref{prop:GeneratorsYoneda}) we are able to prove that $E(\cFK_3)$ is a quotient graded algebra of $S[X]$. We conclude the proof of our main result by showing that the homogeneous components of $E(\cFK_3)$ and $S[X]$ are equidimensional. 

As an application of our main result we also compute the cohomology ring of the bosonization $\cFK_3\#\ku\mathbb{S}_3$. By definition, the bosonization of $\cFK_3$ is the ordinary Hopf algebra whose underlying vector space is $\cFK_3\ot\ku\mathbb{S}_3$ with the algebra and coalgebra structures given by the smash product and smash coproduct, respectively. Our computation is a consequence of the fact that $E(\cFK_3\#\ku\mathbb{S}_3)$ is the invariant ring of a canonical $\mathbb{S}_3$-action on $E(\cFK_3)$. More precisely, in Theorem \ref{te:Bos} we show that $E(\cFK_3\#\ku\mathbb{S}_3)$ is the quotient of the  polynomial ring $\ku[X,U,V]$ by the principal ideal $(U^2V-UV^2)$. A similar computation is performed in Theorem \ref{te:Bos 1} for the dual Hopf algebra of  $\cFK_3\#\ku\mathbb{S}_3$, which is also a bosonization of $\cFK_3$, regarded as a bialgebra in the category of Yetter-Drinfeld modules over the dual of $\ku\Si$.

Finally, we would like to mention that a similar strategy might be helpful for the characterization of the cohomology ring of other bialgebras in categories of Yetter-Drinfeld modules.

\section{Preliminaries}
\begin{fact}[Notation.] Throughout $\ku$ will denote a field of characteristic zero. 
Let $\cX$ be a set. We denote by $\ku \cX$ the vector space freely generated by $\cX$. If $\cX$ is a subset of a $\ku$-algebra $\La$, then $\ku[\cX]$ and $(\cX)$ denote the subalgebra and the ideal generated by $\cX$, respectively. If $\La$ is an $\N$-graded algebra, then $\La_n$ denotes the homogeneous component of degree $n$. For the two-sided ideal $\oplus_{n>0}\La_n$ we shall use the notation $\La_+$. An $\N$-graded algebra $\La$ is said to be connected if and only if $\La_0=\ku$.

If $V$ is a vector space, then let $V^{(n)}:=V\ot\cdots\ot V$, where the tensor product has $n$ factors.
If there is no danger of confusion we shall write  a tensor  monomial $v_1\ot\cdots\ot v_n\in V^{(n)}$ as an $n$-tuple $(v_1,\dots,v_n)$. We shall also use the shorthand notation $v^{(n)}$ for the element $(v,\dots,v)$.
\end{fact}

\begin{fact}[The category \texorpdfstring{$\ydg$}{YD} and Nichols bialgebras.]
Let $G$ be a group. By definition, a Yetter-Drinfeld $\ku G$-module is a vector space $V$ which is endowed with a $G$-action by linear automorphisms and a decomposition as a direct sum of linear subspaces $V=\oplus_{g\in G} V_g$, such that $^gv\in V_{ghg^{-1}}$ for all $g,h\in G$ and $v\in V_h$. A linear map between two Yetter-Drinfeld $\ku G$-modules $V$ and $W$ is called a morphism of Yetter-Drinfeld modules if and only if it commutes with the actions and preserves the decompositions of $V$ and $W$. The category of Yetter-Drinfeld $\ku G$-modules will be denoted by $\ydg$.

The category $\ydg$ is braided monoidal. The tensor product of two objects $V$ and $W$ in $\ydg$ is $V\ot W$, the tensor product of the underlying linear spaces with the diagonal action and the decomposition $(V\ot W)_g=\sum_{x\in G}V_x\ot W_{x^{-1}g}$. The braiding in $\ydg$ is given, for  $v\in V_g$ and $w\in W$, by
\[
 \chi_{V,W}:V\ot W\to W\ot V,\quad \chi_{V,W}(v\ot w)={}^gw\ot v.
\]
An algebra $(\La,m,u)$ in $\ydg$ is an ordinary $\ku$-algebra with the additional property that the multiplication $m$ and the unit $u$ are morphisms of Yetter-Drinfeld modules. Coalgebras in this monoidal category can be defined in a similar way. The comultiplication and the counit of a coalgebra will be denoted by
$\Delta$ and $\varepsilon$, respectively. By an $\N$-graded (co)algebra we mean an ordinary $\N$-graded (co)algebra whose homogeneous components are Yetter-Drinfeld modules. 

An algebra $(\La,m,u)$ in $\ydg$ is braided commutative if and only if $m\chi_{\La,\La}=m$. On the other hand, we shall say that an $\N$-graded algebra $\La$ is braided graded commutative if and only if 
\[
 x\cdot y=m(x\ot y)=(-1)^{pq}m\chi_{\La,\La}(x\ot y)=(-1)^{pq}{}  ( {}^gy\cdot x),
\]
for any $x\in \La_g$ and $y\in \La$ homogeneous of degree $p$ and $q$, respectively.

By definition, a bialgebra in $\ydg$ consists of an algebra $(\La,m,u)$ and a coalgebra $(\La,\Delta,\varepsilon)$ such that $\Delta$ and $\varepsilon$ are morphisms of Yetter-Drinfeld modules. 
A bialgebra in $\ydg$ is said to be $\N$-graded if and only if the underlying algebra and coalgebra structures are graded and they have the same homogeneous components. A bialgebra $\La$ in $\ydg$ is called a Nichols algebra if and only $\La_0$ is $1$-dimensional, $\La_1$ coincides with the space of primitive elements and generates $\La$ as an algebra. 

For any Yetter-Drinfeld $\ku G$-module $V$, the tensor algebra $T(V)$ has a canonical bialgebra structure in $\ydg$. If $I(V)$ denotes the sum of all Yetter-Drinfeld submodules of $T(V)$ which are generated by $\N$-homogeneous elements of degree $\geq2$ and that are ideals and coideals of $T(V)$, then the quotient $\mathfrak{B}(V):=T(V)/I(V)$  is a Nichols algebra. It will be called the Nichols algebra associated to $V$. 

For more details on braided bialgebras and Nichols algebras, that are related to our work, the reader is referred to \cite[Section 2]{MS}. This paper is also useful to understand the relationship between Fomin-Kirillov and  braided bialgebras.
\end{fact}

\begin{fact}[The Nichols algebra  \texorpdfstring{$\B(\cO)$}{B(O}.]
 Let $\mathbb{S}_n$ denote the set of all permutation of  $\{1,\dots,n\}$. For the set of transposition in $\mathbb{S}_n$ we use the notation $\mathbb{T}_n$. 

Let $V(\cO)$ be the Yetter-Drinfeld module which as a linear space  has the basis $\{x_t\}_{t\in\cO}$ and whose $\ku\Si$-action and $\ku\Si$-coaction are given by
$$
{}^{g}x_t=\sgn(g)\,x_{gtg^{-1}}\quad\mbox{and}\quad \rho(x_t)=t\ot x_t
$$
for any $t\in \cO$ and $g\in\Si$. The Nichols algebra $\B(\cO)$ of $V(\cO)$ has been  by investigated intensively by several authors, see for instance \cite{AG2,fominkirilov,grania,MS,Ro}. 

Let us recall some basic facts about $\B(\cO)$ from the above references.
As an algebra, $\B(\cO)$ is generated by the set $\{x_t\mid t\in\cO\}$, subject to the relations
\begin{equation}\label{eq:relations of B}
x_t^2\quad\mbox{and}\quad\sum_{(t',t'')\in \cO(g)} x_{t'}x_{t''},
\end{equation}
where $t$ is an arbitrary transposition, $g$ is any $3$-cycle in $\Si$ and $\cO(g)=\{(t',t'')\in \cO\times\cO\mid t't''=g\}$. The comultiplication of $\B(\cO)$ is uniquely defined such that the generators $x_t$ are primitive elements. Since $\B(\cO)$ is a Nichols algebra, any primitive element is a linear combination of the three generators.

From now on, to simplify the notation, we shall denote the Nichols algebra $\B(\cO)$ by $B$ and for the generators we shall use the notation $a=x_{(12)}$, $b=x_{(23)}$ and $c=x_{(13)}$. 
Note that the square of each generator of $B$ is zero, and 
\begin{align}\label{eq:relations_B}
ab+bc+ca=0=ba+ac+cb.
\end{align}
We define $A=\ku[a,b]$ and $R=\ku[c]$ to be the subalgebras generated by $\{a,b\}$ and $\{c\}$, respectively. 

The subalgebra $A$ can be presented by generators and relations as follows: $A=\Bbbk\langle a,b\mid a^2,b^2,aba-bab\rangle$ and the set $\mathcal{A}=\{1,a,b,ab,ba,aba\}$ is a $\Bbbk$-linear basis of $A$, cf. \cite[Corollary 5.9]{MS}. Clearly, $R=\ku\langle c\mid c^2\rangle$. 
\end{fact}

\begin{fact}[Twisted tensor products.]\label{subsec:TTP}
 We shall see in the next subsection that $B$ is a twisted tensor product of the algebras $A$ and $R$. Since for our work this is one of the most important features of the Fomin-Kirillov algebra $B$, we recall some general facts about arbitrary twisted tensor products of algebras.  For more details on twisted tensor products of algebras see, for example, 
\cite{Cap95a,JPS,Ma,Ta}. 
Let $A$ and $R$ be associative and unital $\ku$-algebras with the product given by the maps  $m_A$ and $m_R$, respectively. By definition a $\ku$-linear map $\sigma :R\otimes A\rightarrow A\otimes R$ is a twisting map between $A$ and $R$ if and only if it is compatible with $m_A$ and $m_R$, that is
\begin{align*}
  \sigma (\id_{R}\otimes m_{A})& =(m_{A}\otimes \id_{R}) (\id_{A}\otimes \sigma ) (\sigma \otimes \id_{A}),  
	\\
  \sigma (m_{R}\otimes \id_{A})& =(\id_{A}\otimes m_{R}) (\sigma \otimes \id_{R}) (\id_{R}\otimes \sigma ).  
\end{align*}%
In addition, $\sigma $ must be compatible with the units of $A$ and $R$, that is $\sigma (r\otimes 1_{A})=1_{A}\otimes r$ and $\sigma (1_{R}\otimes x)=x\otimes 1_{R}$, for all $x\in A$ and $r\in R$. 

To any twisting map $\sigma$ between $A$ and $R$ corresponds an associative and unital algebra structure on $A\ot R$, whose product is given by the relation:
\[
 m_{A\ot R}:=(m_A\ot m_R)(\id_A\ot\,\sigma\ot\id_R).
\]
Of course, the unit of $A\ot R$ is $1_A\ot 1_R$. We shall call this algebra the $\sigma$-twisted tensor product of $A$ and $R$ and we shall denote it by $A\ot_\sigma R$. Throughout the paper we shall assume that the twisting map $\sigma$ is invertible. We denote its inverse by $\tau$.

Let us now assume that $A$ and $R$ are both graded. We shall say that the twisting map $\sigma $ between $A$ and $B$ is graded if it preserves the canonical gradings on $R\ot A$ and $A\ot R$. For such a  $\sigma $ the algebra $A\otimes _{\sigma }R$ is graded with respect to the decomposition $A\ot_\sigma R=\bigoplus_{d\geq 0}\big(\bigoplus_{p=0}^d A^{p}\otimes R^{d-p}\big)$.
In this paper we work with graded twisting maps which, in addition, satisfies the relation $\sigma(R\ot A_+)=A_+\ot R$. The restriction of $\sigma$ to  $R\ot A_+$ will be denoted by $\sigma_+$. For the inverse of $\sigma_+$ we shall use the notation $\tau_+$.

Let us assume that $R$ is finite dimensional. We fix a basis $\{e_1,\dots,e_n\}$ on $R$ such that $e_1=1_R$ and all other elements belong to $R_+$. For any $1\leq i,j\leq n$, there are endomorphisms $\sigma_{ij}$ and $\tau_{ij}$ of $A$ so that 
$$\sigma(e_i\ot x)=\sum_{j=1}^n \sigma_{ji}(x)\ot e_{j}\quad\text{and}\quad \tau(x\ot e_i)=\sum_{j=1}^n e_{j}\ot \tau_{ji}(x).
$$
Note that $\sigma$ and $\tau$ are uniquely defined by the matrices $\widetilde{\sigma}$ and $\widetilde{\tau}$ whose $(i,j)$-entries  are the endomorphisms $\sigma_{ij}$ and $\tau_{ij}$. Since $\tau$ is the inverse of $\sigma$, these matrices are inverses each other. If $x\in A$, then let $\widetilde{\sigma}(x)$ denote the matrix with $(i,j)$-entry $\sigma_{ij}(x)$. 

Since $\sigma$ is compatible with $1_R$ and $1_A$ we get $\sigma_{11}=\id_A$,  $\sigma_{1j}=0$ for $j>1$, and $\widetilde{\sigma}(1_A)$ is the unit matrix of order $n$ with elements in $A$. The compatibility relation between $\sigma$ and the product of $A$ can be written as $\widetilde{\sigma}(x)^t\cdot\widetilde{\sigma}(y)^t=\widetilde{\sigma}(xy)^t$, for any $x,y\in A$, where $\widetilde{\sigma}(x)^t$ denotes the transpose of $\widetilde{\sigma}(x)$. Let $\big\{c_{ij}^k\big\}_{i,j,k}$ be the structure constants  of $R$, that is $e_ie_j=\sum_{p=1}^n c_{ij}^pe_p$.  As $\sigma$ and $m_R$ are compatible, 
\[
 \sum_{p=1}^n c_{ij}^p\sigma_{kp}=\sum_{p,q=1}^n c_{pq}^k\sigma_{pi}\sigma_{qj}.
\]
For any $q>0$, we lift $\sigma$  to a linear map $\sigma^{(q)}:R\ot A^{(q)}\longrightarrow A^{(q)}\ot R$ such that $\sigma^{(1)}=\sigma$ and 
\begin{equation}\label{eq:sigma_recurence}
 \sigma^{(q+1)}=\big(\id_{A^{(q)}}\ot\; \sigma\big)\big(\sigma^{(q)} \ot\; \id_{A}\big).
\end{equation}
Obviously $\sigma^{(q)}$ is invertible. Let $\tau^{(q)}$ be the inverse of  $\sigma^{(q)}$. 
Therefore, if $x=x^1\ot\cdots \ot x^q$, then
\begin{align}\label{Sigma_n}
 \sigma^{(q)}(e_{i}\ot x)=&\sum_{i_1,\dots,i_{q}=1}^n \sigma_{i_1i}(x^1)\ot \sigma_{i_{2}{i_1}}(x^{2})\ot\cdots \ot \sigma_{i_{q}i_{q-1}}(x^{q}) \ot e_{i_{q}},\\\label{Tau_n}
 \tau^{(q)}(x\ot e_{i})=&\sum_{i_1,\dots,i_{q}=1}^n  e_{i_1}\ot \tau_{i_1i_2}(x^1)\ot\cdots \ot \tau_{i_{q-1}i_{q}}(x^{q-1})\ot\tau_{i_{q}i}(x^{q}).
\end{align} 
Since by assumption $\sigma(R\ot A_+)=A_+\ot R$, the endomorphisms $\sigma_{ij}$ and $\tau_{ij}$ map $A_+$ into itself. If $\sigma^{(q)}_+$ is the restriction of  $\sigma^{(q)}$ to $R\ot A^{(q)}_+$, then the image of this map is $A^{(q)}_+\ot R$. The inverse of  $\sigma^{(q)}_+$ will be denoted by  $\tau^{(q)}_+$.
\end{fact}

\begin{fact}[The Nichols algebra  \texorpdfstring{$\B(\cO)$}{B(O)} as a twisted tensor product.]\label{subsec:B is a twist} 
We are going to apply a result due to A. Milinski and H.-J. Schneider, for proving that the Fomin-Kirillov algebra $B$ is a twisted tensor product. We keep the notation from  \cite[Theorem 3.2]{MS}. Let $L'$ and $L$ be the group algebras of the cyclic group generated by $(13)\in\Si$ and of $\Si$, respectively. We choose the braided bialgebras in the categories ${}^{L'}_{L'}\mathcal{YD}$ and ${}^{L}_{L}\mathcal{YD}$ to be $R$ and $B$, respectively. Finally, we take $i:R\to B$ and $\phi:B\to R$ to be the inclusion map and the algebra morphism such that $\phi(a)=\phi(b)=0$ and $\phi(c)=c$.

The Nichols algebra $R$ coacts on $B$ via the algebra map  $\rho:B\to B\ot R$ uniquely defined such that  $\rho(a)=a\ot 1$,  $\rho(b)=b\ot 1$ and $\rho(c)=c\ot 1+ 1\ot c$. The $R$-coinvariant subalgebra $B^{co(R)}$ and $A$ coincide. The multiplication map $m_B$ induces linear transformations  $\nu:A\ot R\rightarrow B$ and $\nu':R\ot A\to B$.  The former map is invertible, cf.  \cite[Theorem 3.2]{MS}. Thus, by \cite[p. 5]{BM}, the map $\sigma:=\nu^{-1}\nu'$ is a twisting map and $\nu$ is an isomorphism of algebras between $B$ and $A\ot_\sigma R$.   

Since any twisting tensor product $A\ot_\sigma R$ is free as a left $A$-module it follows that  $B$ as a left $A$-module is free as well. In particular, we deduce that $B$ is $12$-dimensional, having as a linear basis the set
\[\textstyle
 \mathcal{B}:=\cA\bigcup\cA c=\{1,a,b,ab,ba, aba,c,ac,bc,abc,bac, abac \}.
\] 
We use the basis $\{1,c\}$ of $R$ to compute the twisting map $\sigma$. Let $\widetilde{\sigma}$ be the corresponding matrix with elements in $\End_\ku(A)$, see \S\ref{subsec:TTP}.  For simplicity we shall rename  $\sigma_{12}$ and $\sigma_{22}$ by $\alpha$ and $\beta$. So for any $x\in A$,
\begin{equation}\label{eq:definition of sigma}
\widetilde{\sigma}(x)=
		\begin{bmatrix}
			x &  \alpha(x) \\
			0&\beta(x)
		\end{bmatrix}.
\end{equation}
Taking into account the equation \eqref{eq:relations_B} we have $ca=-ab-bc$. Thus $\sigma(c\ot a)=-1_R\ot ab-c\ot b$. In a similar way  we obtain the relation $\sigma(c\ot b)=-1_R\ot ba-c\ot a$. Hence 
\begin{align}\label{eq:DefinitionAlphaBeta}
  \alpha (a)=-ab,\quad  \alpha (b)=-ba,\quad
   \beta (a)=-b\quad \text{and}\quad  \beta (b)=-a.
\end{align}
The condition  $\widetilde{\sigma}(x)^t\cdot\widetilde{\sigma}(y)^t=\widetilde{\sigma}(xy)^t$, for any $x,y\in A$, means that $\beta$ is an algebra map and $\alpha$ is a $(1,\beta)$-skew derivation. 
Since $a$ and $b$  generate $A$  as an algebra, the above relations uniquely determine $\alpha$ and $\beta$ and imply the following equations:
\begin{align}\label{eq:properties of alpha and beta}
\alpha^2=0,\quad \beta\alpha=-\alpha\beta\quad\mbox{and}\quad\beta^2=1.
\end{align}
The twisting map $\sigma$ is bijective. Indeed, in view of the foregoing remarks, the inverse of $\widetilde{\sigma}$ is the matrix
\begin{equation*}
\widetilde{\tau}=
		\begin{bmatrix}
			\id_A & \beta\alpha\\
			0 &\beta
		\end{bmatrix}.
\end{equation*}
Clearly, $\beta$ is a graded algebra morphism. Although $\alpha$ does not preserve the grading, it maps $A_+$ to itself. It follows that  $\sigma(R\ot A_+)=A_+\ot R$.

Finally, we note that there exist 
$\alpha_n$, $\beta_n$,  $\alpha_n'$ and $\beta_n'$ in $\End_\ku\big(A^{(n)}\big)$ such that the linear liftings $\sigma^{(n)}$ and $\tau^{(n)}$ are given,  for any $x\in A^{(n)}$, by
\begin{equation}\label{ec:Sigma^n}
\sigma^{(n)}(c\ot x)=\alpha_n(x)\ot 1+\beta_n(x)\ot c\quad\text{and}\quad \tau^{(n)}(x\ot c)=1\ot\alpha_n'(x)+c\ot\beta_n'(x).
\end{equation}
By equations \eqref{Sigma_n} and \eqref{Tau_n} we deduce that
\begin{equation}\label{RAction}
 \begin{array}{lll}\displaystyle
 \alpha_{n} =\sum_{i=0}^{n-1}\beta^{\ot i}\ot \alpha\ot \id_{A^{(n-i-1)}}, && \beta_{n}  =\beta^{\ot n};\\[4mm] 
\displaystyle\alpha'_{n}=\sum_{i=0}^{n-1}\id_{A^{(i)}_+}\ot\;\beta\alpha\ot\beta^{\ot (n-i-1)},& & \beta'_{n}  =\beta^{\ot n}.
\end{array} 
\end{equation}
Obviously, $\sigma_+^{(n)}$ and $\tau_+^{(n)}$ also satisfy the relations \eqref{ec:Sigma^n} for any $x\in A_+^{(n)}$. 
\end{fact}

\begin{fact}[The normalized bar resolution.]\label{SubSecBarResolution}
Let $\Lambda$ be a connected graded $\ku$-algebra with associative multiplication $m:\La\ot\La\to\La$.  We denote the restriction of $m$ to $\La_+\ot\La_+$ by $m_+$.

Recall that the normalized bar resolution $P_*(\La)$ of $\ku$ in the category of left $\La$-modules is given by the exact sequence
\begin{align*}
0\longleftarrow \ku\overset{d_0}{\longleftarrow}\La\overset{d_1}{\longleftarrow}\La\ot\La_+\overset{d_2}{\longleftarrow}\La\ot\La_+^{(2)}\overset{d_3}{\longleftarrow}\La\ot\La_+^{(3)}\overset{d_4}{\longleftarrow}\La\ot\La_+^{(4)}\overset{d_5}{\longleftarrow}\cdots
\end{align*}
where $d_0$ is the augmentation map of $\La$. 

For $n>0$ we have  $d_n=\sum_{i=0}^{n-1}(-1)^id_n^i$ and every map $d_n^i$ is induced by $m$. More precisely, $d_n^i(\lambda_0\ot\cdots\ot\lambda_n)=\lambda_0\ot\cdots\ot\lambda_i\lambda_{i+1}\ot\cdots\ot\lambda_n$.

To compute  $E_n(\La,M):=\Ext_\La^n(\ku,M)$ we use the  complex $\Ou{*}{\La,M}$ obtained by applying the functor $\Hom_\La(-,M)$ to $P_*(\La)$. Thus in degree $n$ we have $\Ou{n}{\La,M}=\Hom_\ku(\La_+^{(n)},M)$, where by convention $\Ou{0}{\La,M}=M$.  The differential maps are given by $\lD^0(m)(\lambda)=\lambda\cdot m$ and
\[
 \lD^n(f)(\lambda_0,\dots,\lambda_n)=\lambda_0\cdot f(\lambda_1,\dots,\lambda_n)+\sum_{i=0}^{n-1}(-1)^{i+1}f(\lambda_0,\dots,\lambda_i\lambda_{i+1},\dots,\lambda_n)
\]
for any $n>0$ and $f\in\Ou{n}{\La,M}$. Recall that $(\lambda_0,\dots,\lambda_n)$ is a shorthand notation for  $\lambda_0\ot\cdots\ot\lambda_n$.

If  $M=\ku$, then we shall use the notation $\Ou{*}{\La}:=\Ou{*}{\La,\ku}$. Since the action of $\Lambda$ on $\ku$ is trivial, $\big(\Ou{*}{\La},\lD^*\big)$ is a DG-algebra with respect to the Yoneda product
\[
 f g(\lambda_1,\dots,\lambda_{n+m})=f(\lambda_1,\dots,\lambda_n)\,g(\lambda_{n+1},\dots,\lambda_{n+m}),
\]
where $f\in\Ou{n}{\La}$ and $g\in\Ou{m}{\La}$. The vector space $E(\La):=\oplus_{n\in\mathbb{N}}E_n(\La,\ku)$ inherits an $\N$-graded algebra structure, which will be called the Yoneda ring or the cohomology ring of $\La$. 

In the case when $\La$ is finite dimensional, the complex $\Ou{*}{\La}$ can be rewritten as follows. Let $V$ denote the dual vector space of $\La_+$. Hence $\Ou{n}{\La}\cong V^{(n)}$ and, via this identification, $\lD^1$ can be regarded  as a map from $V$ to $V\ot V$ such that
\[
 \lD^1(f)=-\sum_{i=1}^n f_i'\ot f_i''
\]
if and only if $f(\lambda_1\lambda_2)=\sum_{i=1}^n f_i'(\lambda_1) f_i''(\lambda_2)$, for any $\lambda_1,\lambda_2\in \La_+$. 

Similarly, for $p>1$ the differential map $\lD^p$ is a morphism from $V^{\ot p}$ to $V^{\ot p+1}$satisfying the relation:
\begin{align}
\lD^{p+1}=\lD^1\ot\id_{V^{(p)}}-\id_V\ot\lD^p.
\end{align}
In conclusion, the  $n$-degree component $E_n(\La)$ of $E(\La)$ is the $n$th cohomology group of the complex:
\begin{align}\label{NormalizedBarComplex}
0\longrightarrow\ku\overset{\lD^0}{\longrightarrow}V\overset{\lD^1}{\longrightarrow}V^{(2)}\overset{\lD^2}{\longrightarrow}V^{(3)}\overset{\lD^3}{\longrightarrow}V^{(4)}\overset{\lD^4}{\longrightarrow}\cdots.
\end{align}
It is worthwhile to note that the Yoneda product on $\Ou{*}{\La}$ coincides with the multiplication of the tensor (free) algebra $T(V)$ of the vector space $V$.
\end{fact}

\begin{exa}\label{ex:Omega(R)}
 Let $R=\ku\big\langle c|c^2\big\rangle$. For an $R$-module $M$ we define $D\in\End_\ku(M)$ by $D(m)=c\cdot m$. Since $R_+$ is $1$-dimensional the normalized bar complex $\Ou{*}{R,M}$ can be identified to the complex $(C^*,d^*)$, where $C^n:=M$ and $d^n=D$. Hence $\Ext^0_R(R,M)=\Ker D$ and $\Ext^n_R(R,M)=\Ker D/\Imm D$, for $n>0$.
 
If $M=\ku$ is the trivial $R$-module, then $D=0$. Thus, $E(R)$ is the polynomial ring $\ku[X]$. The indeterminate $X$ is an element of degree $1$, which corresponds to the $1$-cocycle defined by the linear transformation  $f:R_+\to\ku$,  $f(c)=1$.
\end{exa}

\begin{fact}[The normalized bar complex of an algebra in \texorpdfstring{$\ydg$}{YD}.]
We now take $\La$ to be a finite dimensional algebra in the  category of graded Yetter-Drinfeld modules over $\ku G$. Thus  $V$, the dual of $\La_+$, is also an object in this category. Its  component $V_p$ of degree $p$ consists of all linear forms which vanish on $\La_q$ for any positive $q\neq p$, and $V_p$ is an object in the category $\yd{G}$ with respect to the action $(g,f)\longmapsto {}^gf$, where ${^gf}(v)=f\big({}^{g^{-1}}v\big)$. The coaction is defined by the decomposition $V_p=\bigoplus_{g\in G}V_{p,g}$. By definition, $f\in V_{p,g}$ if and only if $f(v)=0$  for any $v\in \La_{p,h}$ such that $h\neq g^{-1}$. For details on the Yetter-Drinfeld module structure of the linear dual of an object in $\ydg$ the reader is referred to \cite[Section 2]{AG}. 
Since $\ydg$ is monoidal, $V^{(n)}:=\bigoplus_{p\in\N} V^{(n)}_p$ is a graded object of this category, where
\[\textstyle
 V^{(n)}_p:=\big\langle v_1\ot\cdots\ot v_n\mid \sum\limits_{i=1}^n\deg v_i=p\big\rangle.
\]
By definition, the $\ku G$-module structure of $V^{(n)}_p$  is induced by the diagonal action, and the component of degree $g\in G$ is spanned by all $ v_1\ot\cdots\ot v_n\in V^{(n)}_p$ such that $v_i\in V_{g_i}$ and  $g_1\cdots g_n=g$. 

The differentials of \eqref{NormalizedBarComplex} are morphism of graded Yetter-Drinfeld modules, so $\Ou{*}{\La}$ is a direct sum 
\[\textstyle
 \Ou{*}{\La}=\bigoplus\limits_{p\in\mathbb{N}}\Ou{*}{\La,p}
\]
of subcomplexes in $\ydg$.
Since $\Ou{n}{\La,p}=V^{(n)}_p$, it follows that  $\Ou{n}{\La,p}=0$ for $n>p$.

Let $\Omega^*(\La,p)_e^{G}$ be the subcomplex of all cochains in  $\Omega^*(\La,p)$ that are of degree $e$ and $G$-invariant. Hence,
$\omega\in\Omega^n(\La,p)_e^{G}$ if and only if 
$^g\omega=\omega$ for all $g\in G$ and $\omega$ belongs to $\Omega^n(\La,p)_e$, the component of  $e$-degree elements of the Yetter-Drinfeld module $\Omega^n(\La,p) $. 
\end{fact}

\begin{fact}[The dimension of  \texorpdfstring{$\Omega^n(\La,p)^G_e$}{O}.]\label{sub:omega-inv} 
Let $V$ be a $\ku$-linear representation of a finite group $G$ and let $V^G$ denote the space of $G$-invariant elements of $V$. Recall that, throughout this paper, the characteristic of $\ku$ is $0$. If $V_G=\ku\ot_{\ku G}V$, then the linear map $\phi:V^G\to V_G$ given by $\phi(v)=1\ot v$ is an isomorphism. Its inverse maps  $1\ot v$ to $\frac{1}{|G|}\sum_{g\in G}{}^gv$.

Let $\cX$ be a $G$-set. We shall denote the orbit and the stabilizer of $x\in \cX$ by $[x]$ and $G_x$, respectively. The set of all orbits of $\cX$ will be denoted $\cX/G$. We fix a set of representatives $\mathcal{R}$ of $\cX/G$, that is a subset of $\cX$ so that any orbit contains a unique element in $\mathcal{R}$. 

By definition a linear representation $V$ of $G$ is called $\cX$-graded provided that it is endowed with a decomposition $V:=\oplus_{x\in \cX}V_x$ such that $^gV_x\subseteq V_{g\cdot x}$, for any $g\in G$ and $x\in \cX$. For example, any Yetter-Drinfeld $\ku G$-module may be seen as an $\cX$-graded module, by taking the $G$-set $\cX$ to be $G$ with the adjoint action. Note that $V_x$ is a $\ku G_x$-module. If $x\in\R$ we define $V_{[x]}:=\oplus_{y\in [x]}V_y$. Clearly,  $V_{[x]}$ is an $\cX$-graded representation of $G$ and we have
$$
V_{[x]}=\ku G\ot{}_{\ku G_x}V_x.
$$
Moreover, since the $\ku$-linear transformation $\phi$ is an isomorphism we get
\begin{align*}\textstyle
V^G=\bigoplus\limits_{x\in\R}V_{[x]}^G\simeq \bigoplus\limits_{x\in\R}\ku\ot{}_{\ku G}{\;}(\ku G\ot{}_{\ku G_x} V_x)\simeq\bigoplus\limits_{x\in\R} \ku\ot{}_{\ku G_x} V_x\simeq\bigoplus\limits_{x\in\R}V_x^{G_x}.
\end{align*}
Hence for any $\cX$-graded representation $V$ we have $$\dim V^G=\sum_{x\in\mathcal{R}}\dim V_x^{G_x}.$$ We are going to use this relation to compute the dimension  of  $\Omega^n(\La,p)_e^G$, where $\La$ is a finite dimensional graded algebra in the category of Yetter-Drinfeld $\ku G$-modules. For, we introduce some more notation.

Let $\mathbb{N}^n[p]$ denote the set of all $n$-tuples $\lp=(p_1,\dots , p_n)$ in $\N^n$  such that all $p_i$ are positive and $\sum_{i=1}^np_i=p$. For a graded vector space $V=\oplus_{n\in\N}V_n$ and an $n$-tuple $\lp=(p_1,\dots,p_n)$ as above let  $V_{\lp}$ denote $V_{p_1}\ot \cdots \ot V_{p_n}$. 
Furthermore, we assume that each $V_n=\oplus_{x\in G}V_{n,x}$ is an  Yetter-Drinfeld $\ku G$-module and we use the notation $V_{\lp,\lx}:=V_{p_1,x_1}\ot\cdots\ot V_{p_n,x_n}$, for any $\lp\in \N^n$ and $\lx\in G^n$.

Let $G$ act on $G^n$ by ${}^g(x_1,\dots,x_n)=(gx_1g^{-1},\dots,gx_ng^{-1})$. Obviously $\cX^{(n)}=\{\lx\in G^n\mid x_1\cdots x_n=e\}$  is a $G$-subset of $G^n$ and the homogeneous component of degree $e$ of $V^{(n)}_p:=\oplus_{\lp\in\N[p]}V_{\lp}$ is the subspace
\begin{align}\label{omega-e}\textstyle
\big(V^{(n)}_p\big)_e=\bigoplus\limits_{\lx\in \cX^{(n)}}\bigoplus\limits_{\lp\in N^n[p]}V_{\lp,\lx}.
\end{align}
We now take $V$ to be the dual of $\La_+$, regarded as a graded object in $\ydg$. Thus,
$$
\dim\Omega^n(\La,p)_e^{G}=\sum_{\lx\in\R}\sum_{\lp\in\N^n[p]}\dim (V_{\lp,\lx})^{G_{\lx}}.
$$
In this relation $\R$ and $G_{\lx}$ denote a set of representatives for $\cX^{(n)}/G$ and the stabilizer of $\lx$, respectively. Note that $G_{\lx}=\bigcap_{i=1}^n C_G(x_i)$, where $C_G(x_i)$ is the centralizer of $x_i$ in $G$.

We say that $\lp$ and $\llq$ are equivalent and we write $\lp\sim\llq$ if and only if  there is $\pi\in\mathbb{S}_n$ such that $q_i=p_{\pi(i)}$ for all $1\leq i\leq n$. Note that $V_{\lp}$ and $V_{\llq}$ are isomorphic as $G$-graded representations, provided that $\lp\sim \llq$. If  $\mathcal{P}_n(p)$ denotes the set of all increasing $n$-tuples $\llq\in\N^n[p]$, then every $\lp\in\N^n[p]$  is equivalent to a unique $\llq\in\mathcal{P}_n(p)$. An element of $\mathcal{P}_n(p)$ will be called positive $n$-partition of $p$.

Let $\cX_{\llq}^{(n)}=\{\lx\in \cX^{(n)}\mid V_{\llq,\lx}\neq0\}$. Obviously, $\cX_{\llq}^{(n)}$ is a $G$-subset of $\cX^{(n)}$. To summarize, if $\R_{\llq}=\R\cap \cX_{\llq}^{(n)}$ and $|\llq|$ is the number of tuples $\lp$ equivalent to $\llq$, then  the following formula holds true:
 \begin{align}\label{omega-inv}
\dim\Omega^n(\La,m)_e^{G}=\sum_{\llq\in\mathcal{P}_n(m)}\sum_{\lx\in\R_{\llq}}|\llq|\dim (V_{\llq,\lx})^{G_{\lx}}.
\end{align}
\end{fact}

\begin{fact}[The Cartan-Eilenberg spectral sequence of a twisted tensor product.]\label{subsec:CE}
The main tool that we use for the computation of the ring $E(B)$ is the Cartan-Eilenberg spectral sequence \cite[Theorem XVI.6.1]{CE}. For the variant of this spectral sequence that we need, the reader is referred to  \cite[\S5.3]{ginzburgkumar}. 

A subalgebra $A$ of $B$ is called normal if and only if $A_+B=BA_+$. Thus, for a normal subalgebra $A$ of $B$ we can define the quotient algebra $\overline{B}:=B/A_+B$. Under the additional assumption that $B$ is a flat  right $A$-module, it follows that $\overline{B}$ acts on $E_q(A)$, and we have a multiplicative spectral sequence: 
\begin{equation}\label{eq:CE} 
 E_2^{p,q}=\Ext_{\overline{B}}^p\big(\ku,E_q(A)\big)\Longrightarrow E_{p+q}(B).
\end{equation}
For applying the above spectral sequence to a specific extension $A\subseteq B$, we need an explicit computation of the $\overline{B}$-action  on the $\Ext$ groups $E_q(A)$. 
Since $B$ is flat over $A$, 
$$
\Ext_A^q(\ku,\ku)\simeq\Ext_B^q(B\ot_A\ku,\ku)\simeq\Ext_B^q(\overline{B},\ku).
$$
By the proof of \cite[Lemma 5.2.1]{ginzburgkumar}, it follows that $E_q(A)$ is a left $\overline{B}$-module with respect to the action induced by the right multiplication of $\overline{B}$ on itself. In conclusion, to describe this action we have to consider an injective $B$-resolution $0\rightarrow\ku\rightarrow I^*$ of $\ku$ and then to transport the $\overline{B}$-module structure of $H^q(\overline{B},I^*)$ on $E							_q(A)$ via the above identification.
In the following lemma we prove that, alternatively, we may work with a special projective $B$-resolution of $\overline{B}$.
\end{fact}

\begin{lema}\label{le:B R bimod resol}
Let $P_*\rightarrow \overline{B}\rightarrow 0$ be a  resolution of $\overline{B}$ by projective left $B$-modules. We assume that each $P_*$ is a $(B,\overline{B})$-bimodule and that the maps that define the resolution are right $\overline{B}$-linear. Then there is an isomorphism of left $\overline{B}$-modules
$E_q(A)\simeq H^q(\Hom_B(P_*,\ku))$.
\end{lema}

\begin{proof}
Note that $\Hom_B(P_*,\ku)$ is a complex of left $\overline{B}$-modules with the action induced by the right $\overline{B}$-action on $P_*$. We conclude by remarking that all objects used in the proof of \cite[Theorem 2.7.6]{We} inherit a left $\overline{B}$-module structure and the maps between them are compatible with these actions.
Thus, proceeding as in the proof of that theorem, one shows that there are $\overline{B}$-linear quasi-isomorphisms
\[
\Hom_B(\overline{B},I^*)\rightarrow\Tot(\Hom_B(P_*,I^*)) \leftarrow\Hom_B(P_*,\ku).\qedhere
\]
\end{proof}

In the remaining part of this section, we shall specialize the spectral sequence \eqref{eq:CE} to the particular case when $B=A\ot_\sigma R$, where $A$ is a connected algebra and $R$ is a graded algebra. The fact that $R$ is not necessarily connected will allow as to work with an arbitrary algebra $R$ which is trivially graded, that is $R=R_0$. Furthermore, we assume that the  twisting map $\sigma$ is invertible  and  maps $R\ot A_+$ to $A_+\ot R$.  We retain the notation from \S\ref{subsec:B is a twist}. Thus,  $\sigma_+$ denotes  the restriction of $\sigma$ to $R\ot A_+$ and for the lifting of $\sigma_+$ to $R\ot A_+^{(q)}$,  defined by the recurrence relation  \eqref{eq:sigma_recurence}, we shall use the notation $\sigma^{(q)}$. The inverses of these maps will be denoted by $\tau$, $\tau_+$ and $\tau_+^{(q)}$, respectively.

We regard $A$ and $R$ as subalgebras of $B$ via the inclusions $x\overset{i_A}{\longmapsto} x\ot 1$ and $r\overset{i_R}{\longmapsto} 1\ot r$. We first check that $A$ is normal in $B$ and that $B$ is a flat right $A$-module.
Since by assumption  $\sigma$ and $\sigma_+$ are bijective, it is easy to see that  $AR=RA$ and $A_+R=RA_+$. Thus,
\[
 A_+B=A_+AR=A_+R=RA_+=RAA_+=ARA_+=BA_+,
\]
so $A$ is normal in $B$. Note that $\overline{B}=B/A_+B=B/A_+R\cong R$.

Under the conditions that we imposed on the twisting map, the right $A$-module $B$ is not only flat, but free. Indeed, the twisting map $\sigma: R\ot A\to A\ot R$ is an isomorphism of right $A$-modules, where $R\ot A$ and $A\ot R$ are regarded as right $A$-modules via $m_A$ and  by restriction of scalars with respect to the inclusion of rings $A\subseteq A\ot_\sigma R=B$.

Let us now describe the action of $\overline{B}\cong R$ on $E_q(A)$. 
When  working with twisted tensor products it is much more convenient to use in computations the  graphical representation of morphisms, similar to that one from the theory of braided monoidal categories, see \cite[Chapter XIV.1]{Ka}. This method is also explained in \cite[Section 4]{JPS}, where it is applied to the investigation of Koszulity of graded twisted tensor products. As general rules, the identity of a vector space will be  drawn as a sole string; the tensor  product and the composition of two morphisms will be represented by horizontal and vertical juxtaposition, respectively. 

Following the above references, we draw the string diagrams representing the most important maps that we work with as in the picture below.  
\begin{equation*}   
 \begin{array}{l} 
   \includegraphics{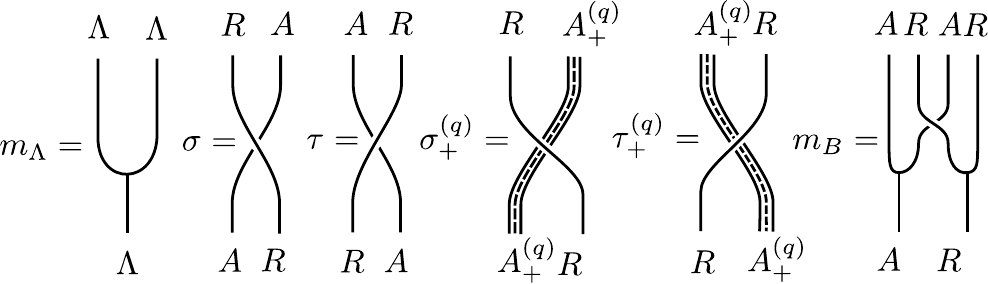} 
 \end{array}
\end{equation*}
In the first three diagrams we introduce the graphical representation of  the multiplication of an algebra $\La$, the twisting map $\sigma$ between $A$ and $R$ and its inverse $\tau$. The next two diagrams represent  $\sigma_+^{(q)}$ and $\tau_+^{(q)}$. Note  the special representation of $\id_{A^{(q)}_+}$ as a `stripe' which replaces the corresponding $q$ strings. 
For the definition of the multiplication  $m_{A\ot_\sigma R}$ see the fifth diagram. In the following two equations  the  compatibility relations between $\sigma$ and the multiplication maps are translated in diagrammatic language. 
\begin{equation*}
 \begin{array}{l}  
   \includegraphics{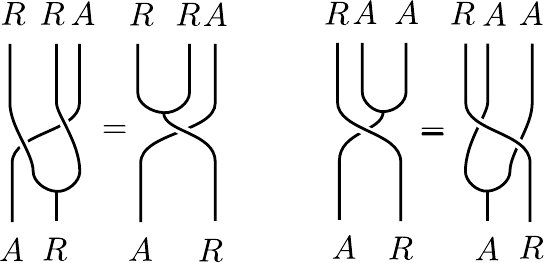} 
 \end{array}
\end{equation*}
Turning back to the Cartan-Eilenberg spectral sequence of a twisted tensor product, we are going to construct a projective resolution of $R$ as in Lemma \ref{le:B R bimod resol}, using the complex $P_*(A)\ot R$.  We endow $A\ot A_+^{(q)}\ot R$ with a $(B,R)$-bimodule structure, where the right and left actions are given  by  $m_R$ and $\mu_q$. For the definition of $\mu_q$  see the first diagram in the figure below. 
\begin{equation}\label{Actions}
 \begin{array}{c}
   \includegraphics{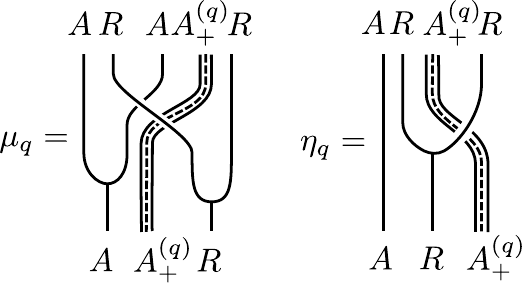}
 \end{array}
\end{equation}
We claim that  $P_*(A)\ot R$ is an exact complex in the category of $(B,R)$-bimodules. Obviously, the complex is exact and the differentials are morphisms of right $R$-modules, since by the definition of the normalized bar resolution  $d_q=\sum_{i=0}^{q-1}(-1)^id_q^i$, cf. \S\ref{SubSecBarResolution}. The fact that $d_q\ot \id_R$ is left $B$-linear is proved by  diagrammatic computation in the next figure. In the first equation we show that $d_q^0\ot \id_R$ is a morphism of left $B$-modules, using first that the product in $A$ is associative and then the compatibility between $m_A$ and $\sigma$. In the second equation we check that $d_q^i\ot \id_R$ is left $B$-linear, for any $0<i<q$.   Here, we need the same compatibility relation once again. 
\begin{equation*} 
 \begin{array}{c} 
 \includegraphics{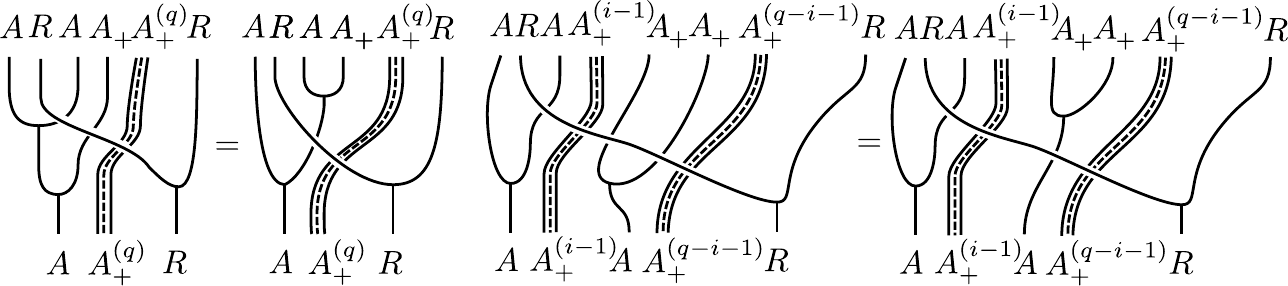} 
\end{array}
\end{equation*}
It remains to show that $A\ot R\ot A_+^{(q)}$ is free as a left $B$-module. Let us remark that $A\ot R\ot A_+^{(q)}$ is a $(B,R)$-bimodule with respect to  the left and right actions given by $m_B$ and $\eta_q$. The latter map is defined in the second diagram of \eqref{Actions}.  

Moreover, by the computation below it follows that the $(B,R)$-bimodule map $\id_A\ot\;\sigma_+^{(q)}:A\ot R\ot A_+^{(q)}\to A\ot A_+^{(q)}\ot R$ is an isomorphism.  In the first equation we prove that this map is a morphism of left $B$-modules using that $\sigma$  and $m_R$ are compatible. The same property and the fact that $\sigma_+^{(q)}$ and $\tau_+^{(q)}$ are inverses each other allow us to conclude that $\id_A\ot\;\sigma_+^{(q)}$ is right $R$-linear as well. 
\begin{center}
 \includegraphics{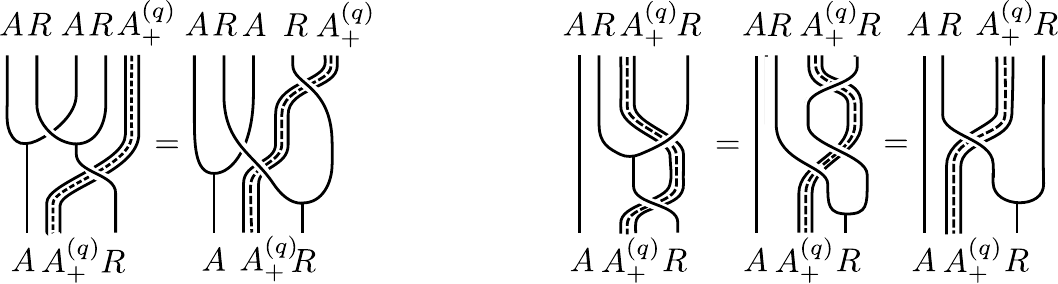}
\end{center}
In particular we deduce that $P_*(A)\ot R$ is an exact complex in the category of $(B,R)$-bimodules which is a resolution of $R$ by free $B$-modules too. In order to get a simpler form of this resolution we define the map $\partial_q=\sum_{i=0}^{q-1}(-1)^i\partial_q^i$, where for every $0\leq i\leq q-1$ we let $\partial_q^i$ to be given by
\[
 \partial_q^i=\big(\id_A \ot\sigma_+^{(q)}\big)\big(d_q^i\ot \id_R\big)\big(\id_A\ot\tau_+^{(q-1)}\big).
\]
Clearly, each $\partial_q^i$ is a $(B,R)$-bimodule morphism.  Obviously $\big(A\ot R\ot A_+^{(*)},\partial_*\big)$ is another resolution of $R$ that shares the properties of $P_*(A)\ot R$. The maps $\partial_q^i$ are computed in the figure below. For future references, we gather the foregoing results in Lemma \ref{le:B R bimod resol explicitly}.
\begin{equation*}
  \begin{array}{c}
    \includegraphics{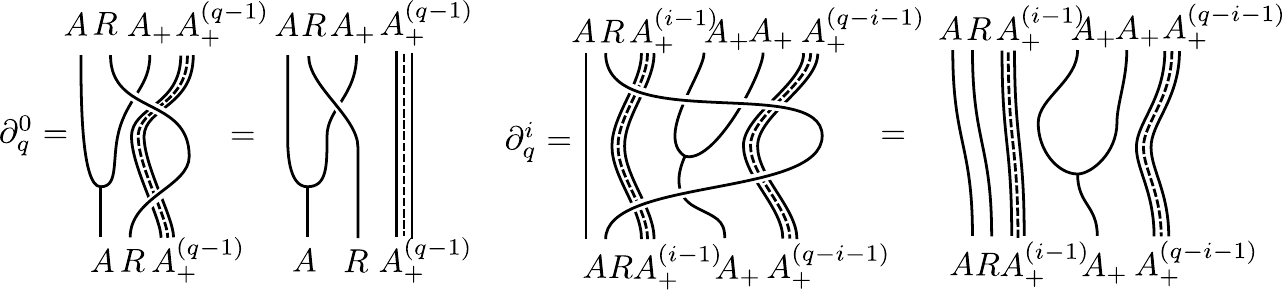}  
  \end{array}
\end{equation*}
\begin{lema}\label{le:B R bimod resol explicitly}
The complex  of  $(B,R)$-bimodules $\big(A\ot R\ot A_+^{(*)},\partial_*\big)$ is a resolution of $R$  by free left $B$-modules. In degree $q$ the differential is defined by the relation   $\partial_q=\sum_{i=0}^{q-1}(-1)^i\partial_q^i$, where
\[ 
 \partial_q^i=\left\{
 \begin{array}{ll}
  \big(m_A\ot\id_{R\ot A\ot A_+^{(q-1)}}\big)\big(\id_A\ot\, \sigma\ot \id_{\ot A_+^{(q-1)}}\big), & \text{if } i=0;\\
 \id_{R\ot A\ot A_+^{(i-1)}} \ot\,  m_A\ot\id_{ A_+^{(q-i-1)}}, &  \text{if } i>0.
 \end{array}\right.
\]
\end{lema}
\begin{obs}
 In the defining relations of $\partial_q^i$, for brevity, we used the  same notation $m_A$ for  the maps induced by the product of $A$. Strictly speaking, if $i=0$, then $m_A$ stands for the restriction of the product to $A\ot A_+$. Similarly, if $i>0$, then $m_A$ denotes the restriction of the multiplication to $A_+\ot A_+$. 
\end{obs}

We have just seen that $E_q(A)$ is the $q$th cohomology group of the complex obtained by applying the functor $\Hom_B(-,\ku)$ to the resolution constructed in Lemma \ref{le:B R bimod resol explicitly}. For any $q$, the resulting vector space  $\Hom_B(B\ot A_+^{(q)},\ku)$ is canonically isomorphic to $\Hom_\ku( A_+^{(q)},\ku)$. Through this identification, the differential $\Hom_B(\partial_{q+1},\ku)$ corresponds  to $-\lD^q$,
where $\lD^q$ denotes the  arrow of  $\Ou{*}{A}$ in degree $q$. 

Via the same identification as above, we get a left $R$-module structure on the $\ku$-dual of $A_+^{(q)}$. By a straightforward computation we can see that, for any $r\in R$, $x\in  A_+^{(q)}$  and  any linear form $f$ on $ A_+^{(q)}$,
\begin{equation}\label{ec:action}
 (r\cdot f)(x)=\sum_{i=1}^m r_{i}\cdot f(x_{i})=\sum_{i=1}^m \varepsilon_R(r_{i})f(x_{i}),
\end{equation}
where $\tau_+^{(q)}(x\ot r)=\sum r_{i}\ot  x_{i}$ and we denoted the augmentation map of $R$ by $\varepsilon_R$.  

We now assume that $A$ and $R$ are finite dimensional. Thus, taking $V$ to be  the linear dual of $A_+$, we can use the complex \eqref{NormalizedBarComplex} for computing $E(A)$. The $R$-module structure of $\Ou{*}{A}$ corresponds to a left $R$-action on this complex. In order to write explicitly the formula for the $R$-action on $V^{(q)}$, we fix a linear basis of $R$, say $\{e_1,\dots,e_n\}$, such that $e_1=1_R$ and all other elements are in $R_+$. 
Using the relation \eqref{Tau_n} we easily see that, for any $f_1,\dots,f_q \in V$,
\begin{equation}\label{ec:ActionOnV^(n)}
e_{i} \cdot(f_1\ot\cdots\ot f_q)=\sum_{i_1,\dots,i_{q-1}=1}^n f_1 \tau_{1i_1}\ot f_2\tau_{i_1i_2}\ot\cdots \ot f_{q-1}\tau_{i_{q-2}i_{q-1}}\ot f_q\tau_{j_{q-1}i},
\end{equation} 
as $e_{j}\in R_+$ for any $j>1$. In conclusion, we have proved the following.  

\begin{theorem}\label{te: CE}
 Let $B=A\ot_\sigma R$ be a twisted tensor product, where $A$ is connected and $R$ is graded, but not necessarily connected. We assume that $\sigma$ is invertible and $\sigma(R\ot A_+)=A_+\ot R$.
 \begin{itemize}
  \item[(a)] The left $R$-module structure of $E(A)$ is induced by the $R$-action on  $\big(\Ou{*}{A}, \lD^*\big)$ given  in \eqref{ec:action}.
  \item[(b)] If $A$ and $R$ are finite dimensional,  then $E(A)$ is the cohomology of the complex \eqref{NormalizedBarComplex}, and the $R$-module structure  of $E(A)$ is induced by the action \eqref{ec:ActionOnV^(n)} on this complex.
  \item[(c)] There is a multiplicative spectral sequence:
  \[
   E^{pq}_2=\Ext_R^p\big(\ku,E_q(A)\big)    \Longrightarrow E_{p+q}(B).
  \]
 \end{itemize}
\end{theorem}

As a first application of Theorem \ref{te: CE}, let us take $B$ to be a twisted tensor product $B=A\ot_\sigma R$, where  $R=\ku \big\langle c\mid c^2\big\rangle$. We keep the notation from \S\ref{subsec:B is a twist} and we choose the basis $\{e_1,e_2\}=\{1,c\}$ on $R$. Thus, in this particular case, for $i=2$, the equation \eqref{ec:ActionOnV^(n)} becomes: 
\begin{equation}\label{ec:ActionDualNumbers}
 c\cdot (f_1\ot\cdots\ot f_q)=\sum_{i=1}^q f_1 \ot\cdots\ot
 f_{i-1} \ot f_i\beta\alpha\ot f_{i+1} \beta \ot\cdots\ot  f_{q} \beta. 
\end{equation}
By Example \ref{ex:Omega(R)} we identify $\Ou{*}{R,E_q(A)}$  to the complex $(C^*,d^*)$, where $C^n=E_n(A)$ and $d^n(\mf{w})= c\cdot \mf{w}$, for any $n$.  For any $p,q\geq 0$ we define
 $Z^{p,q}:=\{ \mf{w}\in E_q(A)\mid c\cdot\mf{w}=0\}$.
On the other hand, we set $B^{0,q}:=0$ and $B^{p,q}:=\{c\cdot \mf{w}\mid \mf{w}\in E_q(A)\}$, for any $p>0$ and $q\geq 0$.
\begin{cor}\label{co:CE}
If $R=\ku \big\langle c\mid c^2\big\rangle$, then there is a multiplicative spectral sequence:
  \begin{equation}\label{ec:CE}
   {E}^{pq}_2 \Longrightarrow E_{p+q}(B),
  \end{equation}
where $E^{pq}_2:=Z^{p,q}/B^{p,q}$, for any $p,q\geq 0$. 
\end{cor}
\begin{proof}
We apply Theorem \ref{te: CE} taking into account the fact that ${E}^{pq}_2$ and $\Ext_R^p\big(\ku,E_q(A)\big)$ are isomorphic.
\end{proof}

\begin{fact}[The Yoneda ring of a smash product.]\label{subsec:bosonization}
Also as an application of the multiplicative spectral sequence from Theorem \ref{te: CE} we shall investigate the relationship between the cohomology ring of a graded $H$-module algebra $\La$ and the cohomology ring of the smash product $\La\# H$. 

Let $H$ be a Hopf algebra with comultiplication $\Delta_H$ and counit $\varepsilon_H$. We assume that the antipode $S$ of $H$ is bijective and we denote the inverse of $S$ by $\overline{S}$. Let us recall that the antipode of a finite dimensional Hopf algebra is always bijective.

A graded  $H$-module algebra consists  of an $\N$-graded algebra  $\La=\oplus_{n\in\N}\La_n$ and an $H$-action on each $\La_n$ such that $h\cdot 1_\La=\varepsilon_H(h)1_\La$ and the following relation holds for any $h\in H, x\in\La_n$ and $y\in\La_m$:
\[
 h\cdot (xy)=(h_{(1)}\cdot x)(h_{(2)}\cdot y).
\]
In the above equation we used Sweedler's notation $\Delta_H(h)={h_{(1)}}\ot h_{(2)}$. As $\La$ is an $H$-module algebra, we can define the smash product $\La\#H$ of $\La$ and $H$. As a vector space $\La\#H=\La\ot H$ and the multiplication is defined by $(a\#h)(b\# k)=a({h_{1}}\cdot b)\#h_{2}k$, for all $x\in \La_n$, $y\in \La_m$ and $h,k\in H$. Of course, the unit of $\La\#H$ is $1_\La\# 1_H$.

The smash product $\La\#H$ can be seen as a graded twisted tensor product with invertible twisting map. Indeed, we define an $\N$-grading on $H$ by imposing that all elements to be of degree $0$. Let $\sigma:H\ot\La\to\La\ot H$ be the $\ku$-linear map defined for $x\in\La_n$ and $h\in H$ by
\[
\sigma(h\ot x)={h_{(1)}}\cdot x\ot h_{(2)}.
\]
One easily sees that $\sigma$ is graded and $\sigma(H\ot\La_+)\subseteq \La_+\ot H$. Moreover, $\sigma$ is bijective and its inverse $\tau$ maps $x\ot h$ to  ${\overline{S}(h_{(1)})}\cdot x\ot h_{(2)}$. 
Thus we can apply Theorem \ref{te: CE} for the twisted tensor product $\La\ot_\sigma H=\La\# H$, so there is a canonical $H$-module structure on $E(\La)$.  
Under the additional assumption that  $H$ and $\La$ are finite dimensional we can compute this action using the relation \eqref{ec:ActionOnV^(n)}, where we take $V$ to be the linear dual of $\La_+$. For, we fix a basis $\{e_1,\dots,e_n\}$ on $H$ such that $e_1=1_H$ and all other elements are in the kernel of $\varepsilon_H$. If  we write $\Delta_H(e_i)=\sum_{j=1}^n e_j\ot h_{ji}$, then the $(i,j)$-entry of the matrix $\widetilde{\tau}$ associated to $\tau$ is the linear endomorphism  $\tau_{ij}:\La\to\La$  which maps $x$ to ${\overline{S}(h_{ij})}\cdot x$. It follows that, for any $f\in V$ and $x\in \La$, 
\[
 f\tau_{ij}(x)=f\big(\overline{S}(h_{ij})\cdot x\big)=(h_{ij}\cdot f)(x),
\]
where the left   $H$-action on $V$ is given by $(h\cdot f)(x)=f(\overline{S}(h)\cdot x)$. By \eqref{ec:ActionOnV^(n)} we get
\[
 e_{i} \cdot(f_1\ot\cdots\ot f_q)=\sum_{i_1,\dots,i_{q-1}=1}^n h_{1i_1}\cdot f_1\ot h_{i_1i_2}\cdot f_2 \ot\cdots \ot  h_{i_{q-2}i_{q-1}} \cdot f_{q-1}\ot  h_{i_{q-1}i} \cdot f_{q}.
\]
Let $\Delta^1_H:=\Delta_H$. For any $i\geq2$ we define the  iterated comultiplication $\Delta_H^{i+1}$ by the recurrence relation $\Delta_H^{i+1}:=(\Delta_H\ot\,\id_{H^{(i)}})\Delta_H^{i}$. In particular we have
\[
 \Delta_H^{q+1}(e_i)=\sum_{i_0,\dots,i_{q-1}=1}^n e_{i_0}\ot h_{i_0i_1}\ot h_{i_1i_2} \ot\cdots \ot  h_{i_{q-2}i_{q-1}}\ot  h_{i_{q-1}i}.
\]
Since $e_1=1$ and $e_{i_0}$ is in the kernel of the counit, for any $i_0>1$, by applying $\varepsilon_H\ot\id_{H^{(n)}}$ to both sides of the above equation, we deduce that 
\[
 \Delta_H^{q}(e_i)=\sum_{i_1,\dots,i_{q-1}=1}^n h_{1i_1}\ot h_{i_1i_2} \ot\cdots \ot  h_{i_{q-2}i_{q-1}}\ot  h_{i_{q-1}i}
\]
In conclusion, the $H$-action on the bar complex \eqref{NormalizedBarComplex} is induced by the diagonal action of $H$ on $V$, defined by $(h\cdot f)(x)=f\big(\overline{S}(h)\cdot x\big)$, for any $h\in H$, $f\in V$ and $x\in \La$. We conclude this section by proving the following result.
\end{fact}

\begin{theorem}\label{te:H-invariants}
 Let $H$ be a semisimple Hopf algebra (a fortiori finite dimensional). If $\La$ is a finite dimensional $H$-module algebra then there is an isomorphism of graded algebras
 \begin{align}\label{eq:yoneda ring of a bosonization}
E(\La\#H)\simeq E(\La)^H, 
\end{align}
where $E(\La)^H=\{\omega\in E(\La)\mid h\cdot \omega=\varepsilon_H(h)\omega, \forall h\in H\}$ is the space of $H$-invariant elements of $E(\La)$ with respect to the $H$-module structure on $E(\La)$  induced by the action on $\Ou{*}{\La}$ given by
\[
 h\cdot (f_1\ot\cdots\ot f_q)=h_{(1)}\cdot f_1\ot\cdots \ot h_{(q)}\cdot f_q,
\]
where for any $h\in H$ and $f\in V$ we have $(h\cdot f)(x)=f(\overline{S}(h)\cdot x)$. In particular, for any finite dimensional algebra in $_H^H\mathcal{YD}$ the smash product $\La\# H$ makes sense, and the isomorphism \eqref{eq:yoneda ring of a bosonization} holds true.
\end{theorem}

\begin{proof}
By the foregoing remarks we know that the $H$-action on $E(\La)$  satisfy the properties stated in the theorem.  As $H$ is semisimple, the multiplicative spectral sequence
\[
 E^{p,q}_2=\Ext_H^p\big(\ku,E_q(\La)\big)\Longrightarrow E_{p+q}(\La\# H)
\] 
(see Theorem \ref{te: CE}) degenerates at the second page, yielding the graded algebra isomorphism \eqref{eq:yoneda ring of a bosonization}. 

Let us now suppose that $\La$ is a finite dimensional algebra in $_H^H\mathcal{YD}$. As the multiplication and the unit of $\La$ are morphisms in the category of Yetter-Drinfeld modules, $\La$ is an $H$-module algebra. Thus  the smash product $\La\#H$ exists, and we have the isomorphism from equation \eqref{eq:yoneda ring of a bosonization}. 
\end{proof}

\begin{obs}\label{obs:bosonization}
If $\La$ is a braided Hopf algebra in the category of Yetter-Drinfeld $H$-modules, then $\La\# H$ has a natural structure of ordinary Hopf algebra, which is called the bosonization of $\La$. As a coalgebra the bosonization of $\La$ is the smash coproduct, see for instance \cite[$\S 1.3$]{AG}.
\end{obs}

\section{The ring  \texorpdfstring{$E(A)$}{E(A)}}\label{section: EA}
Our aim in this section is to investigate the Yoneda ring of the subalgebra $A =\ku[a,b]$ of the $12$-dimensional Fomin-Kirillov algebra $B$. We start by providing a resolution of the trivial left $A$-module $\Bbbk$ by free left $A$-modules. It will be obtained as  the total complex associated to a certain double complex. Then we shall show that the ring $R:=\ku[c]$ acts on the cohomology ring $E(A)$ and we shall explicitly compute this action.
In order to do all that, we consider the following diagram in the category of left $A$-modules, where $\rho_x$ denotes the right multiplication by $x\in A$.
\begin{equation*}
\xymatrix{
\vdots \ar@{->}[1,0]|{\rho_{b} }&
\vdots\ar@{->}[1,0]|{\rho_{ -a}}&
\vdots\ar@{->}[1,0]|{\rho_{b} }&
\vdots\ar@{->}[1,0]|{\rho_{ -a}}&
\vdots\ar@{->}[1,0]|{\rho_{ba} }&\\
A\ar@{->}[1,0]|{\rho_{b} }&
\ar@{->}[0,-1]_{\rho_{ba} }\ar@{->}[1,0]|{\rho_{ -a}}A&
\ar@{->}[0,-1]_{\rho_{ ab}}\ar@{->}[1,0]|{\rho_{b} }A&
\ar@{->}[0,-1]_{\rho_{ba} }\ar@{->}[1,0]|{\rho_{ ab}}A&
\ar@{->}[0,-1]_{\rho_{b} }\ar@{->}[1,0]|{\rho_{ ab}}A&
\cdots\ar@{->}[0,-1]_{\rho_{b}}\\
A\ar@{->}[1,0]|{\rho_{ b}}&
\ar@{->}[0,-1]_{\rho_{ba} }\ar@{->}[1,0]|{\rho_{ -a}}A&
\ar@{->}[0,-1]_{\rho_{ab} }\ar@{->}[1,0]|{\rho_{ba} }A&
\ar@{->}[0,-1]_{\rho_{-a} }\ar@{->}[1,0]|{\rho_{ ba}}A&
\ar@{->}[0,-1]_{\rho_{- a} }\ar@{->}[1,0]|{\rho_{ ba}}A&
\cdots\ar@{->}[0,-1]_{\rho_{-a}}\\
A\ar@{->}[1,0]|{\rho_{b} }&
\ar@{->}[0,-1]_{\rho_{ba} }\ar@{->}[1,0]|{\rho_{ab} }A&
\ar@{->}[0,-1]_{\rho_{b} }\ar@{->}[1,0]|{\rho_{ab} }A&
\ar@{->}[0,-1]_{\rho_{ b}}\ar@{->}[1,0]|{\rho_{ ab}}A&
\ar@{->}[0,-1]_{\rho_{b} }\ar@{->}[1,0]|{\rho_{ab} }A&
\cdots\ar@{->}[0,-1]_{\rho_{b}}\\
A&\ar@{->}[0,-1]^{\rho_{ -a}}A&
\ar@{->}[0,-1]^{\rho_{-a} }A&
\ar@{->}[0,-1]^{\rho_{ -a}}A&
\ar@{->}[0,-1]^{\rho_{-a} }A&
\cdots \ar@{->}[0,-1]^{\rho_{-a}}\\
}
\end{equation*}
We claim that this diagram defines a first-quadrant double complex $C=(C_{\ast\ast},d^h_{\ast\ast}, d^v_{\ast\ast})$, where $C_{p,q}=A$. Note that, if $q>p$ then $d_{p,q}^v=\rho_b$  (respectively $d_{p,q}^v=\rho_{-a}$), provided that  $p$ is even (respectively odd). For $q\geq p$ we have $d_{p,q}^h=\rho_{ba}$  (respectively $d_{p,q}^h=\rho_{ab}$), provided that $q$ is odd (respectively even). On the other hand, for $q<p$ we set $d^h_{p,q}=\beta\circ d^v_{q,p}\circ\beta$ and for any $q\leq p$ one takes $d^v_{p,q}=\beta\circ d^h_{q,p}\circ\beta$.

\begin{lema}
The above diagram is a double complex of left $A$-modules. The total complex  $\Tot_{*}(C)$ is a minimal resolution of  $\Bbbk$ by free left $A$-modules. In particular $\dim E_n(A)=n+1$, for all $n\in\N$. 
\end{lema}

\begin{proof}
The columns and the rows are complexes as $a^2=b^2=0$. The squares anti-commute since $aba=bab$. Let us prove that $\Tot_*(C)$ is exact in positive degree and that its homology in degree $0$ is $\ku$.  By \cite[Definition 5.6.2]{We}, the filtration by rows of $C$ induces a spectral sequence with
$ ^{II}E^1_{p,q}:=H^h_q(C_{p,\ast})$. 
Since the horizontal arrows of $C$ maps the basis $\cA$ to itself, we may identify $ ^{II}E^1_{p,q}$ to a subspace of $A$ which is spanned by some elements of $\cA$.

For $x\in\{a,b\}$ we have  $ \Ker (\rho_x)=\Imm(\rho_x)$. On the other hand, if $x$ and $y$ are distinct elements in  $\{a,b\}$, then
\[
\frac{ A}{\Imm(\rho_{xy})}=\langle 1,x,y,xy\rangle,\quad
\frac{ \Ker (\rho_{yx})}{\Imm(\rho_{xy})}=\langle y,yx\rangle,\quad
\frac{ \Ker (\rho_{xy})}{\Imm(\rho_{x})}=\langle xy\rangle.
\]
Let us denote these vector spaces by $C_{x,y}$, $C'_{x,y}$ and $C''_{x,y}$, respectively. Since $A/\Imm(\rho_a)=\langle 1, b, ab\rangle$ it follows that $ ^{II}E^1_{0,\ast}$ is the complex
\begin{equation*}
0\longleftarrow\langle 1_A, b, ab\rangle\longleftarrow C_{b,a}\longleftarrow C_{b,a}\longleftarrow\cdots \longleftarrow C_{b,a}\longleftarrow\cdots,
\end{equation*}
whose differential maps are all induced by $\rho_b$. It is easy to see that the $n$th homology group of this complex is precisely $\Bbbk 1_A$ provided that $n=0$, and it is trivial otherwise.
Similarly, $^{II}E^1_{1,\ast}$ is the complex
\begin{equation*}
0\longleftarrow 0\longleftarrow C''_{b,a}\longleftarrow C'_{b,a}\longleftarrow\cdots \longleftarrow C'_{b,a}\longleftarrow\cdots.
\end{equation*}
In this case, the map $C'_{b,a}\to C''_{b,a}$ is induced by $\rho_{ab}$, while the other maps are induced by $\rho_{-a}$. It follows that the homology of ${}^{II}E^1_{1,\ast}$ is trivial. For  $p>1$, one proves that  $\h_q(^{II}E^1_{p,\ast})=0$ in a similar way.

We conclude that all $^{II}E^2_{p,q}$ are trivial, excepting $^{II}E^2_{0,0}$ which is of dimension $1$. Thus the spectral sequence collapses at $r=2$ and $\h_n(\Tot_\ast(C))=0$, for any $n>0$. On the other hand $\h_0(\Tot_\ast(C))=\Bbbk$, so $\Tot_\ast(C)$ is a free resolution of $\Bbbk$ in the category of left $A$-modules.

Since $\Hom_A( \rho_x,\Bbbk)=0$ for any $x\in\cA\setminus \{1\}$, it follows that the horizontal and vertical differential maps of the double complex  $\Hom_A( C,\Bbbk)$ are all trivial. In conclusion, the differential maps of $\Hom_A(\Tot_*(C) ,\Bbbk)$ are also trivial. Hence, the resolution  $\Tot_*(C)$ is minimal and
\[
 E_n(A)=\Ext_A^n(\Bbbk,\Bbbk)=\Hom_A(\Tot_*(C) ,\Bbbk)=\Hom_A(A^{n+1},\Bbbk)\cong\Bbbk^{n+1}.
\]
Thus $\dim E_n(A)=n+1$.
\end{proof}

For  the computation of the algebra structure of $E(A)$ we need another important property of $A$, namely that it is a $\mathcal{K}_2$-algebra.
Following \cite{CS} we shall say that  $A$ is a $\mathcal{K}_2$-algebra if and only if $E(A)$ is generated as an algebra by $E_1(A)$ and $E_2(A)$.
Note that any Koszul algebra is $\mathcal{K}_2$, as its cohomology ring is generated by the homogeneous component of degree $1$, cf.  \cite{CS}.

\begin{lema}\label{le:K2}
The algebra $A$ is $\mathcal{K}_2$.
\end{lema}

\begin{proof}
Let $A'=\ku\langle a,b\,|\,a^2,b^2\rangle$. It is well known that  $A'$ is Koszul, see for example \cite[Theorem 6.12]{JPS}. By the foregoing remarks, $A'$ is $\mathcal{K}_2$ as well. Note that $g=aba-bab$ is normal and regular in $A$ and $A=A'/A'g$. We conclude that $A$ is $\mathcal{K}_2$
using \cite[Theorem 9.1]{CS}. Note that the map
$$
 \gamma:\Ext_{A'}^2(\Bbbk,\Bbbk)\to \Ext_{A}^1(\Bbbk,\Bbbk(3))
$$
that appears in the statement of the foregoing mentioned result 
is trivial, in view of the observation preceding \cite[Corollary 9.2]{CS} and the fact that $A'$ has no defining relations of degree $\deg(g)+1=4$.
\end{proof}

\begin{theorem}\label{thm:EA algebra structure}
There is a graded algebra isomorphism
$$
E(A)\simeq\ku\langle \mf{x,y,z}\mid \mf{xy},\mf{yx},\mf{zx}+\mf{yz},\mf{xz}+\mf{zy}\rangle,
$$
such that $\deg \mf{x}=1=\deg \mf{y}$ and $\deg \mf{z}=2$.
\end{theorem}

\begin{proof} We denote the $\ku$-linear dual of  $A_+$ by $V$. We use the normalized bar complex \eqref{NormalizedBarComplex} to find elements $\mathsf{x}$, $\mathsf{y}$ and $\mathsf{z}$ in $E(A)$ that generate this algebra and satisfy the above relations.

If $\cA_+:=\cA\setminus\{1\}$, then we take $\{p_x\}_{x\in\cA_+}$ to be the dual basis of $\cA_+$. It is easy to see that
\begin{align}\label{eq:lD en E1}
\begin{split}
\lD_1(\as)&=0=\lD_1(\bs),\qquad \lD_1(\ab)=\as\ot\bs,\qquad \lD_1(\ba)=\bs\ot\as,\\
&\lD_1(\aba)=\as\ot\ba+\ab\ot\as+\bs\ot\ab+\ba\ot\bs.
\end{split}
\end{align}
We set $\mf{x}$, $\mf{y}$ and $\mf{z}$ to be the cohomology classes of $\as$, $\bs$ and $\bs\ot\ab+\ba\ot\bs$, respectively.
Note that $\mf{x},\mf{y}\in E_1(A)$, while $\mf{z}\in E_2(A)$.  Using the fact that $\dim E_n(A)=n+1$, we deduce that $\{\mf{x},\mf{y}\}$ is a basis of $E_1(A)$ and $\{\mf{x}^2,\mf{y}^2,\mf{z}\}$ is a basis of $E_2(A)$.
By Lemma \ref{le:K2}, the algebra $A$ is $\mathcal{K}_2$. It follows that $E(A)$ is generated as an algebra by $\mf{x}$, $\mf{y}$ and $\mf{z}$.

It remains to prove that these cohomology classes satisfy  the relations from the statement of the theorem. Since $\mf{xy}$ is the cohomology class of $\as\ot\bs$ we get $\mf{xy}=0$ in view of the second relation of \eqref{eq:lD en E1}. The relation $\mf{yx}=0$ can be proved in a similar way.
On the other hand $\mf{xz}+\mf{zy}$ is trivial, as it is the cohomology class of 
\[
w=\lD_2(\aba\ot\bs+\ab\ot\ab).
\]
The last relation can be proved in a similar way.

Let $E:=\ku\langle X,Y,Z\mid XY,YX,XZ+ZY,YZ+ZX\rangle$. We regard $E$ as a graded algebra by imposing the relations $\deg X=\deg Y=1$ and $\deg Z=2$. By the preceding remarks, there is a canonical surjective morphism of graded algebras $\phi:E\to E(A)$.  As $Z\in E_2$, it follows that
\[
\textstyle \cE_n:= \{X^{n-2i}Z^i\mid 0\leq i\leq \frac{n}{2}\}\bigcup \{Y^{n-2i}Z^i\mid 0\leq i\leq \frac{n}{2}\}
\]
is a basis of $E_n$. Since $\dim E_n=n+1$, we conclude that $\phi$ is an isomorphism.
\end{proof}

\begin{lema}
The action \eqref{ec:ActionDualNumbers} of $R$ on the generators of $E(A)$ is given by
$$
c\cdot \mf{x}=0=c\cdot \mf{y}\quad\mbox{and}\quad c\cdot \mf{z}=\mf{x}^2-\mf{y}^2.
$$
\end{lema}

\begin{proof}
Recall that $\mf{x}$ is the cohomology class of $p_a$, see the proof of Theorem \ref{thm:EA algebra structure}. Hence  $c\cdot \mf{x}$ is the cohomology class of $c\cdot p_a$. Note that $c\cdot p_a=-p_a\beta\alpha=0$, since  by \eqref{eq:DefinitionAlphaBeta} the image of $\alpha\beta$ is included into $A_{\geq 2}$,  the square of the augmentation ideal $A_+$, and $p_a$ vanishes on this ideal by definition. It follows that $c\cdot \mf{x}=0$. The equation $c\cdot \mf{y}=0$ can be proved in a similar way.

Let  $w=\bs\ot\ab+\ba\ot\bs$. Taking into account the equation  \eqref{ec:ActionDualNumbers} we have
\begin{align*}
c\cdot w=p_b\ot p_{ab}\beta\alpha+p_b\beta\alpha\ot p_{ab}\circ\beta+p_{ba}\ot p_b\beta\alpha+p_{ba}\beta\alpha\ot p_b\beta.
\end{align*}
We have seen that $p_b\beta\alpha=0$. By equation \eqref{eq:DefinitionAlphaBeta}, we get $p_{ab}\beta\alpha=-p_b$ and $p_{ba}\beta\alpha=-p_a$. On the other hand, $p_b\beta=-p_a$. Therefore
$ c\cdot w=p_a\ot p_a-p_b\ot p_b.$
As $\mf{z}$ is the cohomology class of $w$, we get $c\cdot \mf{z}=\mf{x}^2-\mf{y}^2$.
\end{proof}

\begin{theorem}\label{thm:R action}
Let $\mf{t}\in\{ \mf{x},\mf{y}\}$. The   $R$-action on $E(A)$ satisfies  the relations
\begin{gather*}
c\cdot(\mf{t}^{l}\,\mf{z}^{2k})=0\quad\text{and}\quad c\cdot(\mf{t}^{l}\,\mf{z}^{2k+1})=\mf{t}^{l}(\mf{x}^2-\mf{y}^2)\mf{z}^{2k}.
\end{gather*}
\end{theorem}

\begin{proof}
Clearly, $c\cdot 1=0$, since $\alpha(1)=0$. As $\beta$ is a bijective morphism of algebras, $\beta'_q=\beta^ {\ot q}=\beta_q$ induces an automorphism  of $E_q(A)$ that will be denoted  by $\beta_q$ as well. By \eqref{ec:ActionDualNumbers} we have
\begin{align}\label{Yoneda}
c\cdot(\mf{w'}\mf{w''})=\mf{w'}(c\cdot\mf{w''})+(c\cdot\mf{w'})\,{\beta}_q(\mf{w''})
\end{align}
for any homogeneous elements $\mf{w'},\mf{w''}\in E(A)$. Therefore,  $c\cdot(\mf{w'w''})=\mf{w'}(c\cdot\mf{w''})$, provided that $c\cdot\mf{w'}=0$. Keeping this additional assumption on $\mf{w'}$, it follows by induction that  $c\cdot\mf{w'}^{\hspace*{0.5px}l}=0$, for any $l\geq0$.

Let us prove that $c\cdot \mf{z}^2=0$. Equation \eqref{Yoneda} yields
$$
c\cdot \mf{z}^2=\mf{z}(c\cdot \mf{z})+(c\cdot \mf{z}){\beta}_2'(\mf{z}).
$$
Let $w=\bs\ot\ab+\ba\ot\bs$. By equation \eqref{eq:DefinitionAlphaBeta} we get
\begin{align*}
w(\beta\ot\beta)&=\bs\beta\ot\ab\beta+\ba\beta\ot\bs\beta \\ &= -\as\ot\ba-\ab\ot\as\\
&=\bs\ot\ab+\ba\ot\bs-\delta_1(\aba).
\end{align*}
Hence $w(\beta\ot\beta)=w+\delta_1(\aba)$. This relation implies that  $\beta_2(\mf{z})=\mf{z}$. Since by Theorem \ref{thm:EA algebra structure} the equations $\mf{x}\mf{z}=-\mf{y}\mf{z}$ and $\mf{y}\mf{z}=-\mf{x}\mf{z}$ hold true, we get
$$
c\cdot \mf{z}^2=\mf{z}(\mf{x}^2-\mf{y}^2)+(\mf{x}^2-\mf{y}^2)\mf{z}=0.
$$
Consequently,  $c\cdot \mf{z}^{2k}=0$ for any $k\geq0$.  Furthermore,
$$
c\cdot \mf{z}^{2k+1}=\mf{z}^{2k}(c\cdot \mf{z})={ (\mf{x}^2-\mf{y}^2)\mf{z}^{2k}},
$$
where for the last equality we used the preceding lemma and the relation $\mf{z}^2(\mf{x}^2-\mf{y}^2)=(\mf{x}^2-\mf{y}^2)\mf{z}^2$.
By the same lemma, $c\cdot \mf{t}=0$. Thus, $c\cdot \mf{t}^l=0$ for any $l\geq0$. Since both $ \mf{t}^l$ and $\mf{z}^{2k}$ are annihilated by the action of $c$, we deduce that $c\cdot(\mf{t}^{l}\mf{z}^{2k})=0$. We conclude the proof by the following computation:
\begin{equation*}
c\cdot(\mf{t}^{l}\,\mf{z}^{2k+1})=\mf{t}^{l}\mf{z}^{2k}(c\cdot\mf{z})={\mf{t}^{l}\mf{z}^{2k}(\mf{x}^2-\mf{y}^2)} ={\mf{t}^{l}(\mf{x}^2-\mf{y}^2)\mf{z}^{2k}}. \qedhere
\end{equation*}
\end{proof}
\begin{obs}\label{obs:c-action}
 Let $l>0$. Since $\mf{x}\mf{y}=0=\mf{y}\mf{x}$, the action of $R$ on $E(A)$ satisfies the relations
 \[
 c\cdot(\mf{x}^{l}\,\mf{z}^{2k+1})=\mf{x}^{l+2}\mf{z}^{2k}\quad\text{and}\quad c\cdot(\mf{y}^{l}\,\mf{z}^{2k+1})=-\mf{y}^{l+2}\mf{z}^{2k},
 \]
 for any $n,l\in\N$. 
\end{obs}

\section{The cohomology ring of the 12-dimensional Fomin-Kirillov algebra}

In this section  we shall prove our main result, the computation of the Yoneda ring $E(B)$. We start by describing the subalgebra $\widetilde{E}(B)$ generated by $E_1(B)$.  On the dual vector space of $B_+$, we take the dual basis $\{q_x\}_{x\in\cB_+}$ of $\cB_+=\cB\setminus\{1\}$. Then $E_1(B)$ is spanned by the cohomology classes $[q_a], [q_b]$ and $[q_c]$. In particular, they generate $\widetilde{E}(B)$ as an algebra.  
Some obvious relations between these generators arise from the fact that $\widetilde{E}(B)$ is graded braided commutative.

\begin{lema}\label{BraidedCommutative}
The Yoneda ring $E(B)$ is  graded braided commutative. In particular, $[q_x] [q_y]=[q_z] [q_x]$ for any permutation $(x,y,z)$ of $(a,b,c)$.
\end{lema}

\begin{proof}
Recall that, by definition, a graded algebra $\La$ in the category of Yetter-Drinfeld $\ku{\Si}$-modules is  graded braided commutative if and only if 
$ ^gw\cdot v =(-1)^{nm}{v}\cdot {w}$, for any  ${v}\in\La_n$ and ${w}\in\La_m$ so that $v$ is homogeneous of degree $g$.

By  \cite{MPSW} the cohomology ring of a Nichols bialgebra is graded  braided commutative. Thus, in particular, $E(B)$ is such an algebra, so $ [q_a] [q_b]=-{}^{(12)}[q_b] [q_a]$. Since the $\Si$-action on $E_1(B)$ is induced by that one of the dual of $B$ and $^{(12)}q_b=-q_c$, we have just proved the required relation in the particular case when $(x, y, z)=(a, b,c)$. For any other permutation  $(x, y, z)$ we can prove the corresponding relation in a similar way.
\end{proof}

\begin{obs}
We know that $E_1(B)$ is an Yetter-Drinfeld $\ku\Si$-module. By the proof of the preceding lemma,  the cohomology classes $\a=[q_a]$, $\b=[q_b]$ and $\cc=[q_c]$ are the elements of a vector basis on $E_1(B)$. Moreover, they are homogeneous of the same degree as $a$, $b$ and $c$, the generators of $B$. Since the actions of $\Si$ on $\{\a,\b,\cc\}$ and $\{a,b,c\}$ are identical, it follows that we can identify $E_1(B)$ with the Yetter-Drinfeld module $B_1=V(\T3)$ in such a way that  $\a$, $\b$ and $\cc$ correspond to $a$, $b$ and $c$.
\end{obs}

Our next goal is to show that there are no other relations in the presentation of $\widetilde{E}(B)$. For, we are going to investigate some properties of an arbitrary algebra that contains three elements satisfying the set of relations from Lemma \ref{BraidedCommutative}. More precisely, we are interested in computing the dimension of the homogeneous components of the subalgebra generated by these three elements.

 Let $\La=\oplus_{n\in\mathbb{N}}\La_n$ be a graded $\ku$-algebra. We assume that $\mathcal{S}= \{\a,\b,\cc\}\subseteq \La_1$ is a set  of nonzero elements  such that $ xy=zx$, for any permutation $(x,y,z)$ of $(\a,\b,\cc)$. For such a permutation  $(x,y,z)$
  \[
  xyz=zyx=xzx=zxz=yz^2=z^2y=yx^2=x^2y.
  \]
Let $\Gamma$ denote the subalgebra generated by $\cS$.

\begin{lema}\label{le:basis}
Let $\cS^1=\cS$ and $\cS^n:=\{\a^n,\b^n,\cc^n,\a^{n-1}\b,\a^{n-1}\cc,\a^{n-2}\b^2\}$, for any $n\geq 2$.  The set $\cS^n$ spans linearly the homogeneous component $\Gamma_n$, for any $n>0$.
\end{lema}

 \begin{proof}
For $n=1$ we have nothing to prove. We show by induction on $n\geq 2$ that any monomial $\omega=x_1\cdots x_n$ in $\Gamma_n$ belongs to $\cS^n$. The case $n=2$ and $x_1=x_2$ is clear, as $x_1^2\in\cS^2$. If $x_1\neq x_2$, then let $x_3$ such that $\cS=\{x_1,x_2,x_3\}$. Since $ x_1 x_2=x_2x_3=x_3x_1$ it follows that $\omega=\a u$, for a certain $u\in\cS\setminus\{\a\}$.  

Supposing that we have proved that $ \Gamma_n\subseteq\cS^n$, let us pick a monomial $\omega=x_1\cdots x_{n+1}$ in $\Gamma_{n+1}$. We can assume that at least two factors of $\omega$ are distinct; otherwise, $\omega$ is obviously an element of $\cS^n$. Henceforth, $\omega=x_1^px_{p+1}\cdots x_{n+1}$, with $ x_1\neq x_{p+1}$. Note that $x_1^px_{p+1}=x_{p+1}x_1^p$, if $p$ is even, and $x_1^p x_{p+1}=x_1 x_{p+1}x_1^{p-1}$, if $p$ is odd. Thus $\omega$ can be written as a product of $n+1$ elements of $\cS$, satisfying the additional property that the first two factors are not equal. By the case $n=2$,  we can rewrite $\omega$ such that its first factor is  $\a$. By the induction hypothesis it follows that
\[
 \omega\in\{\a^{n+1},\a\b^n,\a\cc^n,\a^{n}\b,\a^{n}\cc,\a^{n-1}\b^2\}
\]
It remains to prove that $\a\b^n$ and $\a\cc^n$ belong to $\cS^{n+1}$. To show that the former element is in $\cS^{n+1}$ we are going to prove the following two equations
\[
 \a\b^{2p+1}=\a^{2p+1}\b\qquad\text{and}\qquad \a\b^{2p}=\a^{2p-1}\b^2.
\]
Using the relations $x^2y=z^2y=yx^2=yz^2$, we get $\a\b^3=\a\b^2\b=\a\cc^2\b=\a\a^2\b=\a^3\b.$ So, by induction,
\[
 \a\b^{2p+1}=\a\b^{2p-1}\b^2=\a^{2p-1}\b^3=\a^{2p-2}\a\b^3=\a^{2p+1}\b.
\]
The second relation is a consequence of the first one: $\a\b^{2p}=\a\b^{2p-1}\b=\a^{2p-1}\b^2$. The fact that  $\a\cc^n\in\cS^{n+1}$can be proved in a similar way, remarking that $\a\cc^2=\a\b^2$.
 \end{proof}

 \begin{lema} \label{le:DimensionLambda}
 We keep the notation from the above lemma. Let us suppose in addition that:
 \begin{enumerate}
  \item[(i)] $\La$ is a $\mathbb{S}_3$-graded algebra, whose homogeneous component of degree $g$ is $\La_g$.
  \item[(ii)] $\a\in\La_{(12)}$, $\b\in\La_{(23)}$ and $\cc\in\La_{(13)}$;
  \item[(iii)] There is an algebra morphism $\theta:\La\to\ku$ such that $\theta(x)=1$  for all $x\in\cS$;
  \item[(iv)] For every $x\in\cS$ there is an algebra morphism $\theta_x:\La\to\ku$ such that $\theta_x(x)=1$ and $\theta_x(y)=0$, for $y\in \cS\setminus\{x\}$.
 \end{enumerate}
Under these conditions, $\cS^n$ is a basis of $\Gamma_n$, for any $n\geq 1$ (\,$\b^2$ is counted only once in $\cS^2$, of course).
 \end{lema}

\begin{proof}
Let $\omega=x_1\cdots x_n$ be a monomial in $\Gamma_n$. Since $\theta(\omega)=1$ it follows that $\omega\neq 0$. In particular all elements of $\cS^n$ are nonzero. We know that $\cS^n$ generates $\Gamma_n$. Let us assume that $n$ is odd. Then $\{\a^n,\a^{n-2}\b^2\}\subseteq\La_{(12)}$. On the other hand, $\{\b^n,\a^{n-1}\b\}\subseteq \La_{(23)}$ and $\{\cc^n,\a^{n-1}\cc\}\subseteq\La_{(13)}$. Hence $\cS^n$ is linearly independent if and only if the above three sets are so. 

Let $\mu \a^n +\nu \a^{n-2}\b^2=0$ be a linear combination with coefficients in $\ku$. By applying $\theta$ and $\theta_a$ to the left-hand side of this relation we get $\mu=0$ and $\mu+\nu=0$. Hence $\mu=\nu=0$. Proceeding in a similar way one shows that the other two sets are linearly independent.

Now let us assume that $n$ is even and $n\geq 4$ (the case $n=2$ can be handled in an analogous way). For such an $n$ we have $\{\a^n,\b^n,\cc^n,\a^{n-2}\b^2\}\subseteq \La_e$. Since $\a^{n-1}\b\in\La_{(123)}$ and $\a^{n-1}\cc\in\La_{(132)}$ it is enough to show that any linear combination $\mu \a^n+\nu \b^n+\gamma \cc^n +\varepsilon \a^{n-2}\b^2=0$ is trivial. Using the algebra morphisms $\theta$ and $\theta_x$, for all $x\in\cS$, we get the equations $\mu+\nu+\gamma+\varepsilon=0$ and $\mu=\nu=\gamma=0.$
 \end{proof}

As a first application of Lemma \ref{le:DimensionLambda} we consider the following setting. We take $\cS:=\{\a,\b,\cc\}$ to be the set of generators of $E_1(B)$. Let $S$ be the symmetric braided algebra of this Yetter-Drinfeld $\ku\Si$-module, that is the free algebra $\ku\langle\a,\b,\cc\rangle$ modulo the ideal $I$ generated by all relations $xy-zx$ where $x,y$ and $z$ are distinct elements in $\cS$.

Clearly, $S$ is an $\N$-graded algebra. By construction, it is also $\mathbb{S}_3$-graded. Indeed, since $\a,\b$ and $\cc$ are homogeneous of degree $(12),(23)$ and $(13)$, respectively,  it follows that any generator $xy-zy$ of $I$ is also $\mathbb{S}_3$-homogeneous of degree $\deg x\deg y=\deg z\deg x$.

The  canonical map $E_1(B)\to S$ is injective, as $I$ is generated by homogeneous elements of degree $2$. Hence we identify an element  $x\in\cS$ to its image in $S$.

\begin{cor}\label{DimensionLambda^n}
 The set $\cS^n$ is a basis of $S_n$, for any $n\geq 1$.
\end{cor}

 \begin{proof}
The corollary is a direct consequence of the preceding lemma. We have to check the existence of the algebra morphisms $\theta$ and $\theta_x$, for all $x\in\cS$. We define $\theta':E_1(B)\to \ku$ to be the unique linear map sending the elements of $\cS$ to $1$. Using the universal property of $\ku\langle\a,\b,\cc\rangle$, there is an algebra map that lifts $\theta'$. Clearly, this algebra map factorizes through $S$, as it vanishes on the ideal $I$. We take $\theta$ to be the resulting algebra morphism. 
The other three algebra morphisms are constructed in a similar way.
\end{proof}

\begin{prop}\label{prop:La B embeds into EB} We keep the same notation as above.
\begin{enumerate}
 \item[(a)] The algebra  $E(B)$  satisfy the conditions (iii) and (iv) from Lemma \ref{le:DimensionLambda}.
 \item[(b)] There is an embedding of $S$  into $E(B)$ as an algebra in $\yd{\mathbb{S}_3}$ whose image is precisely $\widetilde{E}(B)$.
 \end{enumerate}
\end{prop}

\begin{proof}
Let $x\in B_1$ be a nonzero element such that $x^2=0$. The inclusion $\ku[x]\subseteq B$ induces a morphism of algebras $\vartheta_x:E(B)\to E(\ku[x])$. By Example \ref{ex:Omega(R)} we have $E(\ku[x])\cong \ku[X]$ and the indeterminate $X$ corresponds to the linear function that maps $x$ to $1$. 

For any $1$-cocycle $f:B_+\to \ku$ in $\Ou{*}{B}$ we have $\vartheta_x([f])=f(x)X$. Let $\epsilon:\ku[X]\to \ku$ be the unique algebra map such that $\epsilon(X)=1$. Thus $\epsilon\vartheta_x([f])=f(x)$, so we can take 
 $\theta=\epsilon\vartheta_{a+b+c}$ and $ \theta_x=\epsilon\vartheta_x$ for any $x\in \cS$. Note that $(a+b+c)^2=0$, as the generators of $B$ satisfy the relations \eqref{eq:relations_B}.

To prove the second part, we consider the canonical graded algebra morphism $\varphi: S\to E(B)$ that maps $x$ to $[q_x]$ for any $x\in\cS$.  Obviously, the image of $\varphi$ is $\widetilde{E}(B)$. By Lemma \ref{le:DimensionLambda}, the vector spaces $S_n$ and $\widetilde{E}_n(B)$ are equidimensional, so $\varphi$ is an isomorphism.
\end{proof}

We continue our investigation of the Yoneda ring $E(B)$ by computing  the dimension of the vector spaces ${E}_2^{p,q}$ that define the second page of the spectral sequence from Corollary  \ref{co:CE} (c). 

\begin{prop}\label{le:dim second page}
Let  $m_0=m_2=1$, $m_1=2$ and $m_3=0$. Let $k\geq 0$ denote an integer number.
\begin{enumerate}
 \item[(a)]  For any $i\in \{1,2,3\}$ we have $\dim E_2^{0,4k+i}=2(k+1)$,  and $\dim E_2^{0,4k}=2k+1$.

\item[(b)] For any $p>0$ and $i\in \{0,1,2,3\}$ we have $\dim E_2^{p,4k+i}=m_i$.
\end{enumerate}
\end{prop}

\begin{proof}

In view of Theorem \ref{thm:R action} and the proof of Theorem \ref{thm:EA algebra structure}, it follows that the $R$-action on $E_q(A)$ maps an element of the basis $\mathcal{E}_{q}$ either to zero or to another element of the same basis. Recall that $\mf{z}$ is an element of degree $2$. Hence, it is not difficult to see that
\[
\mathcal{B}'_{p,q}=\big\{ \mf{x}^{i}\mf{z}^{2k} \mid i+4k=q \big\}\bigcup \big\{\mf{y}^{i}\mf{z}^{2k}\mid i+4k=q\big\}
\]
 is a basis of $Z^{p,q}$. Note that $ \mf{x}^{0}\mf{z}^{0}= \mf{y}^{0}\mf{z}^{0}=1$ (the unit of $E_0(A)=\ku$).
By Theorem \ref{thm:R action} and Remark \ref{obs:c-action}, it follows that $ \mf{x}^{i}\mf{z}^{2k}$, $\mf{y}^{i}\mf{z}^{2k}$ and $(\mf{x}^2-\mf{y}^2)\mf{z}^{2j}$  are all in $B^{p,q}$, provided that $i+4k=q=2+4j$ and $i\geq3$.  Since $ \mathcal{E}_q$ is a basis of $E_q(A)$, these elements generate $B^{p,q}$.

 Now we can easily prove  the lemma.  By definition $E_2^{p,q}:=Z^{p,q}/B^{p,q}$. Let $[\mf{w}]_p$ denote the class of $\mf{w}\in Z^{p,q}\subseteq E_q(A)$  in $E_2^{p,q}$. Counting the cardinal of $ \mathcal{B}'_{0,q}$ we conclude the proof of the first part of the proposition, as $E_2^{0,q}=Z^{0,q}$.

 Let $q=0$. Since $E_0(A)=\ku$ and the action of $R$ on $\ku$ is trivial, we get $E_2^{p,0}=\Hom_\ku(R_+^{(p)},\ku)$. Clearly, if $f_p\in E_2^{p,0}$ is the unique linear transformation such that $f_p(c^{(p)})=1$, then $\{f_p\}$ is a basis on $E_2^{p,0}$.
 
 We are now assuming that $p,q>0$. Thus $\mathcal{B}_{p,q}=\mathcal{B}_{p,q}'\bigcup \mathcal{B}_{p,q}''$ is a basis of $E_2^{p,q}$, where
\begin{align*}
  \mathcal{B}_{p,q}'&=\big\{\big[\fij\big]_p\mid i\in\{0,1,2\},\ i+4k=q \big\},\\
  \mathcal{B}_{p,q}''&= \big\{\big[\gij\big]_p\mid i\in\{0,1,2\},\ i+4k=q \big\}.
\end{align*}
 If $q=4k$ then $i=0$ and  $\mathcal{B}_{p,4k}=\big\{[\mf{z}^{2k}]_p\big\}$. In the case when $q=4k+1$, then $\mathcal{B}_{p,q}$ has two elements, $[\mf{xz}^{2k}]_p$ and $[\mf{yz}^{2k}]_p$. Let us suppose that $q=4k+2$. As $\mf{x}^2\mf{z}^{2k}-\mf{y}^2\mf{z}^{2k}$ is an element in $B^{p,q}$,we get $\mathcal{B}_{p,q}=\big\{[\mf{x}^2\mf{z}^{2k}]_p\big\}$. We conclude observing that $\mathcal{B}_{p,q}=\emptyset$, for $q=4k+3$.
\end{proof}

\begin{obs}
 Since  $B_2^{0,q}=0$, any element in $Z_2^{0,q}$ identifies to its class in $E_2^{0,q}$. Thus,  we shall regard $\mathcal{B}_{0,q}$ as a basis of both vector spaces $Z_2^{0,q}$ and $E_2^{0,q}$.
\end{obs}

By definition of  multiplicative spectral sequences, for any $r\geq 2$ the page $E_r^{*,*}=\oplus_{p,q\geq 0}E_r^{p,q}$ is a bigraded  algebra such that the differentials $d_r^{p,q}$ satisfy the graded Leibniz rule. In particular the second page is a bigraded algebra and its component of bidegree $(p,q)$ is $E_2^{p,q}$. Let us remark that $[\mf{w}]_p[\mf{w'}]_{p'}=[\mf{w}\mf{w'}]_{p+p'}$, for any  $\mf{w}\in Z^{p,q}$ and $\mf{w'}\in Z^{p',q'}$.

We have seen that $E_2^{0,*}=\oplus_{q\geq 0}E_2^{0,q}$ is the subalgebra of $E_2^{pq}$ generated by three elements. Two of them, $\mf{f}=[\mf{x}]_0$ and  $\mf{g}=[\mf{y}]_0$, have bidegree  $(0,1)$. The third one,  $\mf{h}=[\mf{z}^2]_0$  is of bidegree  $(0,4)$.

Let us now consider the graded algebra  $E_2^{*,0}=\oplus_{p\geq 0}E_2^{p,0}$. By the proof of Proposition \ref{le:dim second page}, $E_2^{p,0}$ is $1$-dimensional and $f_p$ generates this linear space. Since $f_pf_{p'}=f_{p+p'}$ for any $p,p'\in\N$, we deduce that $E_2^{*,0}$ is the polynomial ring $\ku[\mf{l}]$, where $\mf{l}=f_1$. As an element of $E_2^{*,*}$, the generator  $\mf{l}$ has bidegree $(1,0)$.

\begin{prop}\label{pr:E_2}
 The bigraded algebra $E_2^{*,*}$ is generated by the set $\cE=\{\mf{f},\mf{g},\mf{h},\mf{l}\}$.
\end{prop}

\begin{proof}
We claim that every linear space $E_2^{p,q}$ is contained in the subalgebra generated by $\cE$. Clearly,  $E^{0,q}_2$ is a linear subspace of the subalgebra generated generated by $\cE$, see the foregoing remarks. The claim now follows by remarking that for $p>0$ the set $\mathcal{B}_{p,q}$ is a linear basis of $E_2^{p,q}$ and that the following relations hold: $[\fij]_p=\mf{f}^{ i} \mf{h}^{ k}\mf{l}^{ p}$ and $[\gij]_p=\mf{g}^{ i} \mf{h}^{ k}\mf{l}^{ p}$.
\end{proof}

The next step in the computation of $\dim E_n(B)$ is to  find an upper bound of this number. To do this, we define the sequence $\{N_n\}_{n\geq0}$ of positive integers such that $N_0=1$, $N_1=3$, $N_2=5$ and $N_3=6$. For $n\geq 5$, we define $N_n$ by the recurrence relation:
\begin{equation}\label{def:N_n}
 N_{n+4}=N_n+6.
\end{equation}
Note that $\{N_n\}_{n\geq0}$ contains all integers that are congruent to $1,3,5$ and $0$ modulo 6, that is
$$
\{N_n\}_{n\geq0}=\{1,3,5,6,7,9,11,12,13,\dots\}.
$$

\begin{prop}\label{prop:UpperBoundDimE^n}\label{le:EB hasta 4}
If $n$ is a non-negative integer, then  $\dim E_n(B)\leq N_n$. Moreover, $\dim E_n(B)=N_n$ for $n\leq 3$ and $\dim E_4(B)\geq 6$.
\end{prop}

\begin{proof}
The term $E^{p,q}_{r+1}$ of the spectral sequence from Corollary \ref{co:CE} is a subquotient of $E^{p,q}_{r}$. Thus $\dim E^{p,q}_{r+1}\leq\dim E^{p,q}_{r}$ for any $p,q\geq0$ and $r\geq2$. Since for $r$ large enough we have $\dim E^{p,q}_{\infty}=\dim E^{p,q}_{r}$, we deduce that
$$
\dim E_n(B)=\sum_{p=0}^n\dim E^{p,n-p}_{\infty}\leq \sum_{p=0}^n\dim E^{p,n-p}_{2}.
$$
Let $N'_n:= \sum_{p=0}^n\dim E^{p,n-p}_{2}$. By the above computation, $N_n\leq N_n'$ for all $n\in\N$. We have to prove that  $N'_n=N_n$. Clearly, $N_0=1$. We also have  $N_1=3$ and $N_2=5$, these numbers representing the number of the generators and of the relations that define $B$. Hence, by Proposition \ref{le:dim second page}, it follows that $N_n'=N_n$ for $n=0,1,2$. On the other hand, by Proposition \ref{prop:La B embeds into EB}(c),  the braided symmetric algebra $S$ of $E_1(B)$ embeds into $E(B)$, so $N_3\geq\dim S_3=6$. Since $N_3\leq N_3'=6$ we deduce the relation $N_3=N_3'$. Let us notice that 
\begin{equation*}
 \begin{split}
 N'_{n+4}=\dim E^{0,n+4}_2&+\dim E^{1,n+3}_2+\dim E^{2,n+2}_2+\dim E^{3,n+1}_2+\dim E^{4,n}_2+\\&+\sum_{p=5}^{n+4} \dim E^{p,n-p+4}_{2}.
\end{split}
\end{equation*}
By Proposition \ref{le:dim second page} we have $\sum_{p=1}^{4}\dim E^{p,n-p+4}_{2}=6$. Note that $\dim E^{0,n+4}_2=\dim E^{0,n+2}_2+2$. Since the last term from the right-hand side of the above displayed equation is precisely $\sum_{p=1}^{n}\dim E^{p,n-p}_{2}$ we get $N_{n+4}'=N_n'+6$.  Hence, by induction, $N'_n=N_n$, for all $n$.

In view of Proposition \ref{prop:La B embeds into EB}(c) and Corollary \ref{DimensionLambda^n} we conclude that $N_4\geq\dim S_4=6$.
\end{proof}

\begin{obs}\label{obs:Lambda_B}
The subspace $S_n$ of $E_n(B)=H^n(\Omega^*(B))$ coincides with $H^n(\Omega^*(B,n))$. Thus, as a consequence of the proof of the above lemma, it follows that the $n$th cohomology group of $\Omega^*(B,m)$ vanishes for any $n\leq 3$ and $m\neq n$.
\end{obs}

In order to show that $E_4(B)$ is $7$-dimensional we are going to compute $M_n =\dim \Omega^n(B,6)_e^{\Si}$. Clearly, if either $n<2$ or $n>6$ we have $M_n=0$.

\begin{lema}
If $n\in\{2,3,4,5,6\}$, then $\dim M_n\in \{1,17,68,90,41\}.$
\end{lema}

\begin{proof}
To compute the dimension of these vector spaces we use the relation \eqref{omega-inv}. It is well known that $B$ and its dual $B^*$ are isomorphic as braided algebras in the category of Yetter-Drinfeld $\ku{\Si}$-modules.  Let $B_{n,g}$ denote the set of all  $x\in B_n$ which are homogeneous of degree $g\in\Si$. We use the notation introduced in \S\ref{sub:omega-inv}. For simplicity we shall write $G$ and $\mathcal{P}_n$ instead of $\Si$ and $\mathcal{P}_n(6)$, respectively.  Since $B$ and its dual  are isomorphic as objects in $\yd{G}$  we have $V_{n,g}\cong B_{n,g}$.

Let $n=2$. If $\llq\in\mathcal{P}_2$ and $q_2>4$, then $B_{q_2}=0$. Henceforth, in the relation \eqref{omega-inv} all terms that corresponds to a partition $\llq$ as above vanishes. Therefore, the first sum in the right-hand side of \eqref{omega-inv} has only two terms, corresponding to $(2,4)$ and $(3,3)$. On the other hand, $B_{2,e}=0$ and  $B_{4,g}=0$ for all $g\neq e$. Thus, for $\llq=(2,4)$ and any $\lx\in \cX^{(2)}$ we have $B_{\llq,\lx}=0$, which means that the second sum in the right-hand side of \eqref{omega-inv} is zero.  

Let $\llq=(3,3)$. We do not get new pairs by permuting the components of $\llq$. Thus $|\llq|=1$. It is easy to see that $\cX_{\llq}^{(2)}=\{(hg^i,hg^i)\mid i=0,1,2\}$, where $h=(12)$ and $g=(123)$. If $\lx=(h,h)$, then $G_{h}=\langle h\rangle$. Thus $\cX_{\llq}^{(2)}$ is precisely the orbit $[\lx]$ of $\lx$, and $\R_{\llq}=\{\lx\}$. We deduce that $\dim M_2=\dim(B_{3,h}\ot B_{3,h})^{\langle h\rangle}$. Since, the subspace $B_{3,h}$ is generated by $bac$ and $bac\ot bac$ is $\langle h\rangle$-invariant, we conclude that $\dim M_2=1$.

To show that $\dim M_3=17$ we proceed in a similar way. The positive $3$-partition $\llq$ such that $B_{\llq}\neq0$ are the following: $(1,1,4)$, $(1,2,3)$ and $(2,2,2)$.

If $\llq=(1,1,4)$, then $\cX_{\llq}^{(3)}$ is the set of all triples $(hg^i,hg^i,e)$, where $i=0,1,2$. Hence $|\cX_{\llq}^{(3)}|=3$. Trivially, $|\llq|=3$. In this case there is a unique orbit of length $3$, so $\R_{\llq}=\{(h,h,e)\}$. The stabilizer of the unique element of $\R_{\llq}$ is $\langle h\rangle$ and $(B_{1,h}\ot B_{1,h}\ot B_{4,e})^{\langle h\rangle}=\langle a\ot a\ot abac\rangle_{\ku}$.

If $\llq=(1,2,3)$, then the elements of $\cX_{\llq}^{(3)}$ are the triples $(hg^{i+j},g^i,hg^j)$ with $i=1,2$ and $j=0,1,2$. Thus $|\cX_{\llq}^{(3)}|=6$ and $|\llq|=6$. Let $\lx:=(hg,g,h)$. Obviously, $$G_{\lx}\subseteq C_{G}(h)\bigcap C_{G}(g)=\langle h \rangle\cap\langle g\rangle=\{e\}.$$ Thus $\cX_{\llq}^{(3)}=[\lx]$ and $\dim (B_{\llq,\lx})^{G_{\lx}}=\dim(B_{1,hg}\ot B_{2,g}\ot B_{3,h})=2$.

If $\llq=(2,2,2)$, then $|\llq|=1$ and $\cX_{\llq}^{(3)}$ consists of the triples $(g^{i},g^i,g^i)$, with $i=1,2$. It follows that $|\cX_{\llq}^{(3)}|=2$ and $\R_{\llq}=\{\lx\}$, where $\lx=(g,g,g)$. Thus  $(B_{\llq,\lx})^{G_{\lx}}=(B_{2,g}\ot B_{2,g}\ot B_{2,g})^{\langle g\rangle}$.

As a representation of $\langle g\rangle$ the linear space $B_{2,g}=\langle ab,bc\rangle_{\ku}$ either decomposes as a direct sum of two $1$-dimensional representations or is irreducible, depending on the fact that there is, or there is not, a primitive root of  unity of order $3$  in $\ku$. In the former case, if  $\zeta$ is such a root of unity, it is easy to see that  $(B_{\llq,\lx})^{G_{\lx}}=\bigoplus_{i=1,2}B_{2,g}^{\zeta^i}\ot B_{2,g}^{\zeta^i}\ot B_{2,g}^{\zeta^i}$, where $B_{2,g}^{\zeta^i}=\langle ab-\zeta^ibc\rangle_{\ku}$. Hence $\dim (B_{\llq,\lx})^{G_{\lx}}=2$. On the other hand, if $B_{2,g}$ is irreducible then $B_{2,g}\ot B_{2,g}$ decomposes as a direct sum of $B_{2,g}$ with two copies of the trivial representation. Hence $B_{2,g}^{(3)}$ is a direct sum of three copies of $B_{2,g}$ with two copies of the trivial representation. Thus in this case we also get $\dim (B_{\llq,\lx})^{G_{\lx}}=2$. In conclusion, for $n=3$ we have
$$
 \dim M_3=3\cdot1+6\cdot 2+1\cdot 2=17.
$$
Let $n=4$. If $\llq$ is a positive $4$-partition such that $B_{\llq}\neq0$, then $\llq$ is either $(1,1,1,3)$ or $(1,1,2,2)$. In the first case $|\llq|=4$ and there is an orbit of length $3$ and $4$ orbits of length $6$. The orbit of length $3$ corresponds to $(h,h,h,h)$ and 
$$
(B_{1,h}\ot B_{1,h}\ot B_{1,h}\ot B_{3,h})^{\langle h\rangle}=\langle a\ot a\ot a\ot bac\rangle_{\ku}.
$$ 
If $\lx\in\R_{\llq}$ and $\lx$ is not $(h,h,h,h)$, then $G_{\lx}$ is trivial and $\dim (B_{\llq,\lx})^{G_{\lx}}=1$. For $\llq=(1,1,2,2)$ we get $|\llq|=6$. In this case there are two orbits of length $6$.  For each $\lx\in\R_{\llq}$ the stabilizer $G_{\lx}$ is trivial and we have $$
\dim (B_{\llq,\lx})^{G_{\lx}}=\dim(B_{1,x_1}\ot B_{1,x_2}\ot B_{2,x_3}\ot B_{2,x_4})=4.
$$
We get
$$\dim M_4=1\cdot 4\cdot1+4\cdot4\cdot1+2\cdot6\cdot4=68.$$
For $n=5$ there is a unique positive $5$-partition $\llq$ such that $B_{\llq}\neq 0$, namely $\llq=(1,1,1,1,2)$. Thus $|\llq|=5$. The $G$-set $\cX_{\llq}^{(5)}$ has $9$ orbits, each one of length $6$. If $\lx\in\R_{\llq}$, then $G_{\lx}$ is trivial and $\dim (B_{\llq,\lx})^{G_{\lx}}=2$, as $B_{2,x_5}$ is $2$-dimensional. Therefore, $$\dim M_5=9\cdot5\cdot2=90.$$
Finally, if $n=6$, then there is a unique positive $6$-partition of $6$, namely $\llq=(1,1,1,1,1,1)$. The elements of the $G$-set  $\cX^{(6)}_{\llq}$ are the $6$-tuples $\lx=\big(x_6x_5x_4x_3x_2,x_2,\dots,x_6\big)$, where $x_2,\dots,x_6$ are all transpositions.  Hence $|\cX^{(6)}_{\llq}|=3^5=243$. If $x_2=\cdots = x_6=h$, then $G_{\lx}=\langle h\rangle$ and $(B_{\llq,\lx})^{G_{\lx}}=\langle a^{(6)}\rangle_{\ku}$. On the other hand, if the components of $\lx$ are not equal, let us say that $x_i\neq x_j$, then $G_{\lx}\subseteq\langle x_i\rangle\cap\langle x_j\rangle$. Thus the stabilizer of $\lx$ is trivial and $\cX^{(6)}_{\llq}$ has $40$ orbits of this type, each one of length $6$. Therefore, in this case, $(B_{\llq,\lx})^{G_{\lx}}=B_{1,x_1}\ot \cdots\ot B_{1,x_6}$. We conclude  that $\dim M_6=1\cdot1\cdot1+1\cdot40\cdot1=41$.
\end{proof}

\begin{lema}\label{Generator4}
There is a unique (up to multiplication by a scalar)  class $\mf{d}\in E_4(B)_e$ which is  $\Si$-invariant and does not belong to $S\subseteq E(B)$. In particular, $\dim E_4(B)=7$.
\end{lema}

\begin{proof}
With the notation from the above lemma, the complex $\Ou{*}{B,6}_e^{\Si}$ can be written as follows:
$$
0\overset{}{\longrightarrow}0\overset{}{\longrightarrow}0\overset{}{\longrightarrow} M_2\overset{\delta_6^2}{\longrightarrow} M_3\overset{\delta_6^3}{\longrightarrow} M_4\overset{\delta_6^4}{\longrightarrow} M_5\overset{\delta_6^5}{\longrightarrow} M_6\overset{}{\longrightarrow} 0.
$$
By  Remark \ref{obs:Lambda_B}, it is exact in degree $2$ and $3$, so $\delta_6^2$ is injective and $\dim \Imm\delta_6^3=\dim M_3-\dim M_2=16$.
On the other hand, the sequence
$$
0\overset{}{\longrightarrow} \Ker(\delta_6^5) \overset{}{\longrightarrow} M_5\overset{\delta_6^5}{\longrightarrow} M_6\overset{}{\longrightarrow} H^6(M_*)\overset{}{\longrightarrow} 0
$$
is exact. We have seen that the sixth cohomology group of $\Ou{*}{B,6}$ identifies to $S_6$. Via this identification, a cohomology class in $H^6(M_*)$ corresponds to an $\Si$-invariant homogeneous element of degree $e$  in $S_6$. It is easy to see that $(S_6)^{\Si}_e$  is  spanned by $\a^6+\b^6+\cc^6$ and $\a^4\b^2$, so it $2$-dimensional. Thus $\dim \Ker(\delta_6^5)= 90-41+2=51$, so the dimension of the image of $\delta_6^4$ is at most $51$. We get that $ \dim\Ker\delta_6^4\geq\dim M_4-\dim51=17$, so $\dim H^4(M_*)\geq 1$. We conclude the proof using Remark \ref{obs:Lambda_B} and Proposition \ref{le:EB hasta 4}.
\end{proof}

\begin{theorem}\label{thm:dimE^n}
 We have $\dim E_n(B)=N_n$, for all $n\in\N$.
\end{theorem}

\begin{proof}
By Proposition \ref{le:EB hasta 4} and Lemma \ref{Generator4} it follows that the claimed relation holds for $n\leq4$.
In order to prove that it is true for all $n$, by the proof of Proposition \ref{prop:UpperBoundDimE^n}, it is enough to show that the Cartan-Eilenberg spectral sequence degenerates at the second page, i.e. $E_\infty^{p,q}=E_2^{p,q}$, for any $p,q\geq0$. In other words, we have to prove that all maps $d_r^{p,q}$  vanish for $r\geq 2$.

Let $r=2$. Since $d_2^{*,*}$ satisfies  the graded Leibniz rule, these maps are  trivial if and only if $d_2^{1,0}, d_2^{0,1}$ and $d_2^{4,0}$ vanish on the generators of $E_2^{*,*}$, that is  the following relations hold true
\[
 d_2^{1,0}(\mf{f})=d_2^{1,0}(\mf{g})= d_2^{4,0}(\mf{h})=d_2^{0,1}(\mf{l})=0.
  \]
 Obviously, by definition of spectral sequences,  the maps $d_2^{1,0}$ and  $d_2^{0,1} $ are zero. If we suppose that $d_2^{4,0}(\mf{h})$ is not zero,  then $\dim E_3^{4,0}<\dim E_2^{4,0}$. As in the proof of Proposition \ref{prop:UpperBoundDimE^n}, it would follow that $\dim E_4(B)<7$, fact that is not possible. In conclusion, $d_2^{p,q}=0$, for any $p,q>0$. In particular,  we have an isomorphism of bigraded algebras  $E_3^{*,*}\simeq E_2^{*,*}$, so $E_3^{*,*}$ is generated by elements of degree at most $4$. Thus, we  can repeat the above argument to show by induction that $d_r^{p,q}=0$ for any  $r\geq2$ and all $p,q\geq0$.
\end{proof}

In order to prove that the algebra $E(B)$  is generated by elements of degree at most $4$, we need the following result, which probably is well known. For completeness, we include a proof of it.

\begin{lema}
Let $H=\oplus_{n\geq0}H^n$ be a graded algebra. We assume that each $H^n$ is filtered
$$
{H^n=F^0H^n\supseteq F^1H^n\supseteq \cdots\supseteq F^nH^n\supseteq 0}
$$
such that $(F^pH^q)(F^rH^s)\subseteq F^{p+r}H^{q+s}$.
\begin{enumerate}
\item[(a)] The linear spaces $F^pH=\oplus_{n\geq0}F^pH^n$ defines a decreasing  algebra filtration of $H$ such that the graded associated $\gr_FH$ $=\oplus_{p\geq0}F^pH/F^{p+1}H$ is a bigraded algebra with the homogeneous component of bidegree $(p,q)$ given by $(\gr_FH)^{p,q}=F^pH^{p+q}/F^{p+1}H^{p+q}$.

\item[(b)] If $\gr_FH$ is generated as an algebra by some homogeneous elements $a_1,\dots,a_r$ of degree at most $n$, then $H$ is also generated by some homogeneous elements $h_1,\dots, h_r$ of degree at most $n$.
\end{enumerate}
\end{lema}

\begin{proof}
The first part is an easy exercise. Let us prove (b). Note that the  component of degree $n$ of  $\gr_FH$ is  $\gr^n_FH=\oplus_{p=0}^nF^pH^{n-p}/F^{p+1}H^{n-p}$. If $n_i=\deg a_i$, we may assume that $a_i\in F^{p_i}H^{n_i-p_i}/F^{p_i+1}H^{n_i-p_i}$. Let $h_i$ be an electronic in $F^{p_i}H^{n_i-p_i}$ such that $h_i+ F^{p_i+1}H^{n_i-p_i}=a_i$.

Let $E^{p,q}=(\gr_FH)^{p,q}=F^pH^{p+q}/F^{p+1}H^{p+q}$. Then $\gr^n_FH=\oplus_{p+q=n}E^{p,q}$. We also define the graded algebra $E^{*,0}=\oplus_{p\geq0}E^{p,0}=\oplus_{p\geq0}F^pH^p$. Since $E^{p,q}\cdot E^{p',q'}\subseteq E^{p+p',q+q'}$ and $a_1,\dots, a_r$ generate $E^{*,*}$, it follows that $E^{*,0}$ is generated by those elements $h_i=a_i$ that belong to $E^{p_i,0}$.

We claim that any element in $H$ is a linear combination of monomials in the elements $h_1,\dots,h_r$.
We fix $n$ and we will prove by induction on $i$ that the linear space $F^{n-i}H^n$ is spanned by monomials in $h_1,\dots,  h_r$.
If $i=0$, this follows by the foregoing remark. Let us assume that $F^{n-i}H^n$ satisfies the required property. Let $x\in F^{n-i-1}H^{n}$. If $x\in F^{n-i}H^{n}\subseteq F^{n-i-1}H^{n}$, we have nothing to prove. Otherwise, the class of $x$ in $E^{n-i,i}$ is a linear combination of monomials  $x+F^{n-i}H^n=\sum\alpha_{i_1,\dots,i_k}a_{i_1}\cdots a_{i_k}$. Hence $x-\sum\alpha_{i_1,\dots, i_k}h_{i_1}\cdots h_{i_k}\in F^{n-i}H^n$. We conclude by applying the induction hypothesis.
\end{proof}

\begin{prop}\label{prop:GeneratorsYoneda}
 The $\ku$-algebra $E(B)$ is generated by four elements of degree at most $4$.
\end{prop}

\begin{proof} 
By the proof of Theorem \ref{thm:dimE^n}, the Cartan-Eilenberg spectral sequence collapses at the second page, so the algebras $E_\infty^{*,*}$ and $E_2^{*,*}$ are isomorphic. Thus, by Proposition \ref{pr:E_2}, the algebra $E_\infty^{*,*}$ is generated by elements of degree at most $4$. Now we can apply the preceding lemma, as the spectral sequence  converges to $E(B)$.
\end{proof}

\begin{theorem}\label{te:main}
 The graded algebra $E(B)$ is isomorphic to the polynomial ring $S[X]$, where the grading on the latter ring is taken so that $\deg X=4$.
\end{theorem}

\begin{proof}
Let $H_{S}$ be the Hilbert series of $S$. By Corollary \ref{DimensionLambda^n}, we have
$$H_{S}=1+3t+5t^2+6\sum_{n\geq 3}t^n.$$
On the other hand, the Hilbert series of $\Bbbk[X]$ is $H=\sum_{n\geq 0}t^{4n}$, as $X$ is of degree $4$. Hence $H_{S[X]}$, the Hilbert series of $S[X]\cong S\ot\, \Bbbk [X]$, is  the product of $H_{S}$ and $H$. Thus $H_{S[X]}=\sum_{n\geq 0}N_nt^n$. Henceforth, the homogeneous components of $S[X]$ and $E(B)$ are equidimensional.

Recall that $E(B)$ is braided graded commutative, see Lemma \ref{BraidedCommutative}. Therefore, any $\Si$-invariant cohomology class of even degree is a central element in $E(B)$. In particular, by Lemma \ref{prop:GeneratorsYoneda}, it follows that $\mf{d}$ is central. Let $\Phi:S[X]\to E(B)$ be the unique algebra morphism that extends the identity of $S$ and  maps $X$ to $\mf{d}$. Since $E_n(B)=S_n$, for $n\leq 3$, and $E_4(B)=S_4\oplus \Bbbk\mf{d}$ it follows that $\oplus_{n\leq 4}E_n(B)$ is a subspace of the subalgebra $S[\mf{d}]$ generated by $S$ and $\mf{d}$. By Proposition \ref{prop:GeneratorsYoneda} we deduce that $\oplus_{n\leq 4}E_n(B)$ generates $E(B)$ as an algebra, so $\Phi$ is surjective. In view of the fact that $E_n(B)$ and $S_n$ have the same dimension, we conclude that $\Phi$ is an isomorphism.
\end{proof}

As an application of the above theorem, we compute the cohomology ring of the bosonization of $B$, that is of the ordinary Hopf algebra whose underlying algebra and coalgebra structures are the smash product algebra $B\#\ku\Si$ and the smash coproduct coalgebra $B\#\ku\Si$ (see \S\ref{subsec:bosonization} and Remark \ref{obs:bosonization}).   

By Theorem \ref{te:H-invariants} there is an isomorphism $E(B\#\ku\Si)\cong E(B)^{\Si}$ of graded algebras. Here the invariant subring is taken with respect to the canonical $\Si$-action on $E(B)$ which is obtained  by using the fact that $\Ou{*}{B}$ is a complex in the category of Yetter-Drinfeld $\ku\Si$-modules. By our main result we identify $E(B)$ and $S[X]$. Since $X$ corresponds to the $\Si$-invariant cohomology class $\d$, a polynomial in $S[X]$ is invariant if and only if its coefficients are so.  Thus we have to investigate the properties of  the ring $S^{\Si}$.

\begin{lema}
Let $S$ denote the braided symmetric algebra of $E_1(B)$. Then:
\begin{enumerate}
 \item[(a)]The algebra $S^{\,\Si}$ is generated by $\mf{u}=\a(\b+\cc)$ and $\mf{v}=\a^2+\b^2+\cc^2$.
 \item[(b)] The generators $\mf{u}$ and $\mf{v}$ commute and  $2\mf{u}\mf{v}^2+3\mf{u}^2\mf{v}-9\mf{u}^3=0$.
 \item[(c)] The Hilbert series of $S^{\,\Si}$ is $H_{S^{\,\Si}}=1+2t^2+3\sum_{n\geq 2}t^{2n}$.
 \item[(d)] There is an isomorphism  $S^{\,\Si}\cong\ku[U,V]/(U^2V-UV^2)$ of graded algebras.
\end{enumerate}
\end{lema}

\begin{proof}
To prove (a), let us recall that $S_1$ and $B_1$ are isomorphic as Yetter-Drinfeld modules. Thus $^{(12)}\a=-\a$ and $^{(12)}\b=-\cc$. It follows that for any invariant element $\mf{w}=\alpha\a+\beta\b+\gamma\cc$ in $S_1$ we have $\alpha=0$ and $\beta=\gamma$. On the other hand, $^{(23)}\a=-\cc$ and $^{(23)}\b=-\b$. Since $\mf{w}$ is also $(23)$-invariant, we deduce in a similar way that $\beta=0$. Hence there are no invariant elements in $S_1$, excepting the trivial cohomology class.

We now want to prove that $S_n^{\,\Si}=0$, for any odd number $n>1$. By  Corollary \ref{DimensionLambda^n} the set 
$$
 \cS^n=\{\a^n,\b^n,\cc^n, \a^{n-1}\b,\a^{n-1}\cc,\a^{n-2}\b^2\}
$$ 
is a linear basis of $S_n$. Let $\mf{w}=\alpha\a^n+\beta\b^n+\gamma\cc^n+\delta \a^{n-1}\b+\epsilon \a^{n-1}\cc+\lambda \a^{n-2}\b^2$ be an invariant element in $S_n$. Since $^{(12)}\mf{w}=\mf{w}$, we get $\alpha=-\alpha$ and $\beta=-\gamma$. On the other hand, $\delta=-\epsilon$ and $\lambda=-\lambda$. Note that for the former relation we used the equation $\a^{n-2}\cc^2=\a^{n-2}\b^2$ which, in turn, follows by $\a\cc^2=\a\b^2$; see the proof of Lemma \ref{le:basis}. By imposing the condition $^{(23)}\mf{w}=\mf{w}$ we also get $\beta=0$, so $\mf{w}=\delta\a^{n-1}(\b-\cc)$. Thus  $^{(23)}\mf{w}=\delta(\cc^{n-1}\b-\cc^{n-1}\a)$. Since $\cc^{n-1}\a=\a^{n-2}\b^2$ and $\cc^{n-1}\b=\a^{n-1}\b$ we conclude that $\mf{w}=0$.

Proceeding in a similar way one proves that $\mf{u}$ and $\mf{v}$ span $S_2^{\,\Si}$. Also, if $n$ is even and $n>2$, then we can show that $\a^{n}+\b^n+\cc^n$, $\a^{n-1}(\b+\cc)$ and $\a^{n-2}\b^2$ are linearly independent elements which span $S_n^{\;\Si}$. Note that the Hilbert series of $S^{\Si}$ is given by $H_{S^{\Si}}=1+2t^2+\sum_{n\geq2}3t^{2n}$, so (c) has been proved too.

To complete the proof of (a) we first show by induction on $n$ that the following relations hold true:
\begin{align*}
 \mf{u}^n&=\frac{1}{3}[2^n+(-1)^{n+1}]\a^{2n-1}(\b+\cc)+\frac{2}{3}[2^{n-1}+(-1)^n]\a^{2n-2}\b^2;\\
 \mf{v}^n&=\a^{2n}+\b^{2n}+\cc^{2n}+3(3^{n-1}-1)\a^{2n-2}\b^2;\\
 \mf{u}\mf{v}^n&=3^{n}\a^{2n+1}(\b+\cc).
\end{align*}
These equations imply that, for any $n\geq 2$, the above basis on $S_n^{\Si}$ is included into the subalgebra generated by $\mf{u}$ and $\mf{v}$. Since $\a^2$, $\b^2$ and $\cc^2$ are central elements of $S$ we get $\mf{uv=vu}$. Then the part (b) follows by the computation below:
$$
\mf{u}\mf{v}^2=9\a^5(\b+\cc)=3\mf{u}^3-6\a^4\b^2=3\mf{u}^3-2\mf{v}\a^2\b^2=3\mf{u}^3-\frac{1}{3}(3\mf{u}^2\mf{v}-\mf{u}\mf{v}^2).$$
In order to prove (d), we observe that there is a surjective morphism of graded algebras from $S'=\ku[U,V]/(2UV^2+3U^2V-9U^3)$ to $S^{\Si}$; see the first two parts of the lemma. Since the generator of the ideal that defines $S'$ equals $(3U+V)(-3U+2V)U$, by the variable change 
\[
 U\mapsto 3U+V\qquad\text{and}\qquad V\mapsto -\frac{3}{2}U+V,
\]
 we can identify the graded algebras $S'$ and $\ku[U,V]/ (U^2V-UV^2)$. Thus, via this identification, the generators $U$ and $V$ of $S'$ satisfy the relation $U^2V=UV^2$. Clearly $\dim S_0'=1$ and $\dim S_1'=2$. On the other hand, for $n\geq 2$, the set $\{U^n,V^n, UV^{n-1}\}$ is a basis on $S_n'$. In conclusion the Hilbert series of $S'$ and $S^{\Si}$ are equal, so the canonical graded algebra map from $S'$ to $S^{\Si}$ is an isomorphism.
\end{proof}

\begin{theorem}\label{te:Bos}
 The $\ku$-algebra $E(B\#\ku\Si)$ is isomorphic to $\ku[X,U,V]/(U^2V-VU^2)$, where $\deg U=\deg V=2$ and $\deg X=4$.
\end{theorem}

\begin{proof}
 We have already noticed that by Theorem \ref{te:H-invariants} there is an isomorphism of graded algebras $E(B\#\ku\Si)\cong S[X]^{\Si}$. Since $X$ is an $\Si$-invariant element, a polynomial in $S[X]$ is invariant if and only if it belongs to the polynomial ring $S^{\Si}[X]$.
\end{proof}

\begin{fact}[The Yoneda ring of the bosonization of $B$ by $\ku^{\Si}$, the dual Hopf algebra of $\ku\Si$.]
The Fomin-Kirillov algebra $B$ is a braided graded Hopf algebra in ${}^{\ku^{\Si}}_{\ku^{\Si}}\mathcal{YD}$ with respect to the action and coaction given by
$$
f\cdot x=f(\deg x)\,x\quad\mbox{and}\quad\rho(x)=\sum_{g\in\Si}\delta_{g^{-1}}\ot{}^gx
$$
for all $f\in\ku^{\Si}$ and $x\in B$, where in the second formula $\{\delta_g\}_{g\in\Si}$ denotes the basis of $\ku^{\Si}$ dual to $\{g\}_{g\in\Si}$. It is well-known that $B$ is the Nichols algebra of $B_1$ in the category of Yetter-Drinfeld $\ku^{\Si}$-modules, cf. \cite[$\S$ 3.2]{AV}. We also remark that  $B\#\ku\Si$ and $B\#\ku^{\Si}$ are Hopf algebras dual each other.

We shall compute $E(B\#\ku^{\Si})$ proceeding  as in the proof of Theorem \ref{te:Bos}. Since the counit of $\ku^{\Si}$ is the evaluation map $f\mapsto f(e)$ for all $f\in\ku^{\Si}$, the invariant subring $E(B)^{\ku^{\Si}}$ coincides with the subring of elements of degree $e$. Also, the degree of the generator $X$ of $E(B)=S[X]$ is $e$. Thus we have to compute the subring $S^{\ku^{\Si}}=S_e$.
\end{fact}

\begin{lema}\label{le:Bos 1}
Let $S$ denote the braided symmetric algebra of $E_1(B)$. Then:
\begin{enumerate}
 \item[(a)]The algebra $S_e$ is generated by $\mf{p}=\a^2$, $\mf{q}=\b^2$ and $\mf{r}=\cc^2$.
 \item[(b)] The generators $\mf{p}$, $\mf{q}$ and $\mf{r}$ commute,  and  $\mf{p}\mf{q}=\mf{p}\mf{r}=\mf{q}\mf{r}$.
 \item[(c)] The Hilbert series of $S_e$ is $H_{S_e}=1+3t^2+4\sum_{n\geq 2}t^{2n}$.
 \item[(d)] There is an isomorphism $S_e\cong\ku[P,Q,R]/(PQ-PR, PQ-QR)$ of graded algebras.
\end{enumerate}
\end{lema}

\begin{proof}
Recall that there is an isomorphism $S_1\simeq B_1$  of Yetter-Drinfeld modules. Thus $\deg\a=(12)$, $\deg\b=(23)$ and $\deg\cc=(13)$. 

By Lemma \ref{le:basis}, the set  $\{\a^{2n},\b^{2n},\cc^{2n},\a^{2n-2}\b^2\}_{n\in\N}$ is a basis of $S_{e}$ . Therefore (a) and (c) are clear. The second statement follows by the computation: $$\mf{p}\mf{q}=\a^2\b^2=\a(\a\b^2)=\a(\a\cc^2)=\mf{p}\mf{r}=(\a^2\cc)\cc=(\b^2\cc)\cc=\mf{q}\mf{r}.$$
Let $S'=\ku[P,Q,R]/(PQ-PR, PQ-QR)$. By (a) and (b), there is a surjective morphism of graded algebras from $S'$  to $S_e$. Then to prove (d) it is enough to see that the homogeneous components of  $S'$ and  ${S_e}$ are equidimensional. Clearly $\dim S_0'=1$ and $\dim S_1'=2$. Moreover, for $n\geq 2$, the set $\{P^n, Q^n, R^n, P^{n-1}Q\}$ is a basis of $S_n'$ as, by induction,   $P^nQ=PQ^n=PR^n=Q^nR=QR^n$. 
\end{proof}

The proof of the next theorem is similar to that one of Theorem \ref{te:Bos}.

\begin{theorem}\label{te:Bos 1} 
The $\ku$-algebras $E(B\#\ku^{\Si})$ and $\ku[X,P,Q,R]/(PQ-PR , PQ-QR)$ are isomorphic, where $\deg P=\deg Q=\deg R=2$ and $\deg X=4$. \qed
\end{theorem}

\begin{obs}
For any connected graded coalgebra $C$ over a field $\ku$ one defines the cohomology ring $E(C)=\Ext_C^*(\ku,\ku)$ as in the case of connected algebras, but now the $\Ext$-groups are computed for the trivial right $C$-comodule $\ku$. Since $B\#\ku\Si$ and $B\#\ku^{\mathbb{S}_3}$ are dual each other, it follows that the cohomology ring of the former Hopf algebra, regarded as a coalgebra, is precisely the polynomial algebra from Theorem \ref{te:Bos 1}. 
\end{obs}

\section*{Acknowledgments.} 
The first author was financially supported by UEFISCDI,  Grant 253/05.10.2011  (UEFISCDI code ID 0635).  The second author was partially supported by ANPCyT-Foncyt, CONICET and Secyt (UNC).

This work was carried out in part during the visit of the first author to the University of C\'ordoba (Argentina) under the framework of the program \textit{Internships for foreign researchers~/~experts during sabbatical periods}, supported by CONICET Argentina. He would like to thank the Faculty of Mathematics, Astronomy and Physics for its worm hospitality and support. Both authors are grateful to Nicol\'as Andruskiewitsch for drawing their attention to the intricate problem of computing the cohomology ring of braided bialgebras in general  and Fomin-Kirillov algebras in particular, and also for so many stimulating discussions.

\end{document}